\documentclass[final]{amsart}
 
\usepackage{amsmath}
\usepackage{mathrsfs}
\usepackage{amssymb}
\usepackage{amsfonts}
\usepackage{stmaryrd}
\usepackage{enumerate}
\usepackage[usenames]{color}
\usepackage[pdftex]{graphicx}
\usepackage{graphicx}
\usepackage{epstopdf}
\usepackage[colorlinks=true, pdfstartview=FitV, linkcolor=blue, citecolor=blue, urlcolor=blue,pagebackref=false]{hyperref}
\usepackage{textcomp}

\usepackage{multirow}
\usepackage[margin=4cm]{geometry}

\newtheorem{theorem}{Theorem}[section]
\newtheorem{lemma}[theorem]{Lemma}
\newtheorem{proposition}[theorem]{Proposition}

\renewcommand{\phi}{\varphi}

\newcommand{\R}{\mathbb R}

\newcommand{\Z}{\mathbb Z}
\newcommand{\C}{\mathbb C}

\newcommand{\grad}{\nabla}
\newcommand{\cout}[1]{}

\def \bfo {\begin {eqnarray*} }
\def \efo {\end {eqnarray*} }
\def \ba {\begin {eqnarray*} }
\def \ea {\end {eqnarray*} }
\def \beq {\begin {eqnarray}}
\def \eeq {\end {eqnarray}}
\def \supp {\hbox{supp}\,}

\newcommand{\tmop}[1]{\ensuremath{\operatorname{#1}}}

\title[Inverse problems of coupled-physics imaging methods]{Applications of CGO solutions on coupled-physics inverse problems
    }

\author{ilker Kocyigit, Ru-Yu Lai, Lingyun Qiu, Yang Yang \and Ting Zhou
}

\begin{document}

\begin{abstract}
This paper surveys inverse problems arising in several coupled-physics imaging modalities for both medical and geophysical purposes.  These include Photo-acoustic Tomography (PAT), Thermo-acoustic Tomography (TAT), Electro-Seismic Conversion, Transient Elastrography (TE) and Acousto-Electric Tomography (AET). These inverse problems typically consists of multiple inverse steps, each of which corresponds to one of the wave propagations involved. The review focus on those steps known as the inverse problems with internal data, in which the complex geometrical optics (CGO) solutions to the underlying equations turn out to be useful in showing the uniqueness and stability in determining the desired information.
\end{abstract}

\maketitle

\section{Introduction}

Coupled-physics inverse problems arise in various hybrid medical imaging and seismic imaging modalities. Usually two or more different types of wave propagations are involved, subsequently triggered through natural energy conversion. Such physical coupling mechanism overcomes limitations of classical single-measurement based tomography techniques and delivers potentially life-saving diagnostic information with both better contrast and higher resolution.

To be more specific, many traditional single-propagation-based imaging methods suffer from either low contrast or low resolution. An example of low-contrast imaging method is Ultrasound Imaging (UI). UI exhibits high resolution due to its hyperbolic nature, or more plainly saying, its richer time-dependent measurements. Yet the reconstructed sound speed of the material does not distinguish healthy tissues from cancerous tissues very well since both tissues have similar acoustic properties. On the other side of the spectrum, a class of methods, such as Optical Tomography (OT), Electrical Impedance Tomography (EIT) and so on, aim at reconstructing optical/electrical properties of the material. Such properties are more sensitive to the intrinsic physiological properties (oxy- and deoxy-hemoglobin, water, lipid, and scatter power), hence provide better contrast in imaging soft tissues. However, due to the diffusive nature of such propagations, when the measurements are made outside the object (non-invasively), sharp features of the material have been ``smoothed out", resulting in a low resolution. In mathematical terms, this low resolution phenomena manifests the ill-posedness of the inverse problem of reconstructing diffusive (optical/electrical) coefficients from boundary measurements. The fix offered by multi-wave coupled-physics modalities is roughly speaking to carry the {\it internal information}, correlated to the optical/electrical properties, stably to the boundary using another wave propagation, e.g., the sound wave.

To be more illustrative about the idea, we compare EIT and the most popular coupled-physics methods known as Photo-acoustic Tomography (PAT) and Thermo-Acoustic Tomography (TAT). In EIT, imaging is based on recovery of the value of conductivity function $\gamma(x)$ everywhere in the bounded region $\Omega$ modeling the human organ. The measurement is the voltage-to-current (or current-to-voltage) map taken on the boundary $\partial\Omega$. The mathematical inverse problem is the classical Calder\'on problem \cite{Cal}: to reconstruct $\gamma$ from the Dirichlet-to-Neumann (DtN) map $\Lambda_\gamma$ of the elliptic conductivity equation $\nabla\cdot\gamma\nabla u=0$ (physically Ohm's law). A lot of work has been done in solving this nonlinear inverse problem (see \cite{U2009} for a thorough review of the problem). It has been shown that when $\gamma$ is scalar and satisfies certain regularity conditions, one can expect to reconstruct $\gamma$ uniquely from $\Lambda_\gamma$. However, it is also known (e.g., in \cite{A1988}) that such problem is ill-posed, which accounts for the low resolution behavior mentioned above.  It is well understood that this is due to the smoothing effect of the operator $\gamma\mapsto\Lambda_\gamma$.

On the other hand, PAT and TAT are based on the photo-acoustic effect \cite{CAB, FSS, XW}. When an object (usually animal tissue) is exposed to a short pulse of electromagnetic radiation, a fraction of the radiation is absorbed by the medium, resulting in a thermal expansion. This expansion then emits acoustic waves, which propagate to the boundary of the domain. This physical coupling between the absorbed radiation and the emitted acoustic waves is called the {\it photo-acoustic effect}. What distinguishes PAT and TAT is the frequency of the radiation used to illuminate: in PAT, high frequency radiation such as near-infrared with sub-$\mu$m wavelength is used; while in TAT, low frequency microwave with wavelengths comparable to 1m is used \cite{LPKW}. The inverse imaging process consists of two steps: first, to reconstruct the absorbed radiation inside the medium from acoustic signals measured on the boundary; second, to reconstruct the optical property (exhibiting better contrast) of the tissue from this internal information obtained in the first step. The first step is shown to be a stable inverse source problem for the wave equation (see section \ref{sec:QPAT} for more details and references). The second step is an inverse problem with {\it internal data}, which is richer compared to the {\it boundary data} used in EIT.

From above examples, we observe that coupled-physics imaging methods involve solving multiple steps of inverse problems, each of which has to be well-posed (stable) to give an overall high resolution in imaging. It is a major feature of these methods that usually the last step is to solve an {\em inverse problem with internal data} obtained from previous steps. Such inverse problems are what we focus on in this review. 
The coupled-physics methods we are going to consider include PAT, TAT, Electro-Seismic Conversion (ESC), Transient Elastrography (TE) and Acousto-Electric Tomography (AET). The internal data, which we denote by $H$ through the paper, obtained in these modalities are usually polynomials of the solutions $u$ to the underlying equations of radiation, or their derivatives. In another word, the information of the coefficients of interest are hidden in the form of the underlying equations and the internal data of solutions to them. Interpreted this way, it is not surprising that a particular type of solutions play a major role in solving this type of inverse problems. Here we explore one of such solutions known as complex geometrical optics (CGO) solutions.\\


CGO solutions were first introduced in \cite{SU1987} to the conductivity operator $\nabla\cdot\gamma\nabla$, to solve the nonlinear Calder\'on problem for EIT. Since then, such solutions were successfully constructed to several other equations such as the elasticity equation, Maxwell's equations and so on to solve various inverse problems. (see \cite{COS2009, CP, NU1, NU2, OPS,OS, UW2007, UW2008}).
The strategy of construction usually starts with the reduction to a Schr\"odinger equation $(\Delta+q)u=0$ in three or higher dimensions. In case of systems, Maxwell's equations can be reduced to a matrix Schr\"odinger equation of the same type in a less straightforward fashion (see \cite{OS}) while elasticity equations can be reduced to a Schr\"odinger equation with external Yang-Mills potentials (see \cite{ike}). Given a complex vector $\zeta\in\C^n$ such that $\zeta\cdot\zeta=0$, a CGO solution to $(\Delta+q)u=0$ is of the form
\[u_\zeta=e^{i\zeta\cdot x}(a(x)+\psi_\zeta(x))\]
where $\psi_\zeta$ satisfies certain decaying property as $|\zeta|\rightarrow\infty$. In \cite{SU1987}, this is done by solving an equation about $\psi_\zeta$ with leading operator $\Delta+2i\zeta\cdot\nabla$ whose inverse is the integral operator with Faddeev kernel, that is,
\[G_\zeta(f)=\mathcal F^{-1}\left(\frac{\mathcal F f}{|\xi|^2+2\zeta\cdot\xi}\right).\]
CGO solutions with nonlinear complex phase are also available (see \cite{DKSU, KSU}) using Carleman estimate.
The construction can be manipulated to give solutions vanishing on part of the boundary, which is useful for partial data problems, that is, the inverse problem with measurements taken only on part of the boundary \cite{Bukhgeim,KSU}. For a thorough review on CGO solutions, we refer the reader to the review \cite{U2009}. Recently, breakthrough has been made on constructing CGO solutions to equations with less regular parameters, using an averaging technique, see \cite{CR,Ha,HT}.


In this survey, our emphasis is on applications of CGO solutions in solving inverse problems with internal data arising in various imaging modalities introduced above. It is not an easy task to categorize the modalities based on the usage of CGOs. Instead, we contribute roughly each section to address the application in one modality. The following Table \ref{table:summary} summarizes the modalities, along with the brief information of the underlying equations, formats of the available data, types of CGO solutions used and the obtained results. 
\begin{table}[htp]
\caption{}
\label{table:summary}
\begin{center}
\begin{tabular}{p{2cm}|p{4cm}|p{4cm}|p{2.5cm}}
{\bf Modality} & {\bf Equation and Data}& {\bf CGO solution} & {\bf Results} \\\hline
\multirow{2}{2cm}
{Section \ref{sec:QPAT}: Quantitative PAT (second step of PAT)}
& 
$-\nabla\cdot\gamma\nabla u+\sigma u=0$ 

data: 
$u|_{\partial\Omega}\mapsto \sigma u|_{\Omega}$ 

(full boundary illuminations).
& 
$u = e^{i \zeta\cdot x}(1+\psi_{\zeta})$ 

to the reduced equation $(\Delta+q)u=0$. 
& 
{Uniqueness and stability in determining $(\gamma,\sigma)$ (see \cite{BU2010}).}\\
\cline{2-4}
&
data: 
$u|_{\Gamma}\mapsto\sigma u|_{\Omega}$ 

(partial boundary illuminations).
& 
$u=e^{\frac{1}{h}(\varphi+i\psi)}\big(a+r\big)+z$ 

with 

$\supp~u|_{\partial\Omega}\subset\Gamma$.
&
Uniqueness and stability in determining $(\gamma,\sigma)$ (see \cite{CY2012}).\\ \hline
\multirow{1}{2cm}
{Section \ref{sec:ESC}: Electro Seismic Conversion}
&
Maxwell's equations: 

$\begin{array}{l}\nabla\times E = i\omega\mu_{0} H,\\
\nabla\times H = (\sigma - i\varepsilon\omega)E.\end{array}$ 

data: 

$\nu\times E|_{\partial\Omega}\mapsto LE|_{\Omega}$ 
%
&
$E= e^{i\zeta\cdot x}(\eta + R_\zeta)$
&
Uniqueness and Stability in determining $(L,\sigma)$ (see \cite{CY2013}).\\ \hline
\multirow{1}{2cm}
{Section \ref{sec:TE}: Transient Elastography}
&
Elasticity system: 

$\nabla\cdot (\lambda(\nabla\cdot u)I+2S(\nabla u)\mu)+k^2 u=0$ 

data: 

$u|_{\partial\Omega}\mapsto u|_{\Omega}$.
&
$U = e^{i\zeta\cdot x} (C_0(x,\theta) p(\theta\cdot x)+O(\tau^{-1}))$ 

to the Schr\"odinger equation with external Yang-Mills potentials. 
&
Uniqueness and Stability in determining the Lam\'e parameters $(\lambda,\mu)$ 

(see \cite{Liso}).\\ \hline
\multirow{2}{2cm}
{Section \ref{sec:AET}: Acousto Electric Tomography}
&
\underline{Step 1}: 

Conductivity equation

$\nabla \cdot (\gamma \nabla u) = 0$. 

data: 

$m\mapsto(\Lambda_{\gamma_m}-\Lambda_{\gamma})(u|_{\partial\Omega})$ 

where $\Lambda_\gamma$ is the Dirichlet to Neumann map for $\gamma$ and $\gamma_m=(1+m)\gamma$.
&
$u={\gamma}^{-1/2}e^{i\zeta \cdot x} (1 + \psi_{\zeta})$ 

to the conductivity equation. 
&
Reconstruction of $\sqrt{\gamma}\nabla u|_{\Omega}$ using CGOs (see \cite{ilker2012}) or $\gamma |\nabla u|^2|_{\Omega}$ (see {\cite{BBMT2011}}). \\
\cline{2-4}
&
\underline{Step 2}:

data: 
$u|_{\partial\Omega}\mapsto \sqrt{\gamma}\nabla u|_{\Omega}$ 
 \qquad or \qquad
 $u|_{\partial\Omega}\mapsto \gamma |\nabla u|^2 |_{\Omega}$ 

%
& 
same as step 1 above
&
Uniqueness and stability in determining $\gamma$ 

(see \cite{BBMT2011,ilker2012}).\\ \hline
\multirow{2}{2cm}
{Section \ref{sec:QTAT}: Quantitative TAT}
&
Scalar Schr\"odinger 

$(\Delta+q)u=0$ 

where $q=k^2+ik\sigma(x)$. 

data: 
$u|_{\partial\Omega}\mapsto \sigma|u|^2_{\Omega}$ 
%
&
$u = e^{ i \zeta \cdot x}( 1 + \psi_\zeta)$
&
Uniqueness and Stability in determining $\sigma$ 

(see \cite{BRUZ}).\\
\cline{2-4}
&
Maxwell system:

$-\nabla\times\nabla\times E+qE=0$ where $q=k^2n+ik\sigma$. 

data: $\nu\times E|_{\partial\Omega}\mapsto \sigma|E|^2|_{\Omega}$
& 
$E=\gamma_0^{-1/2} e^{i\zeta\cdot x}\big(\eta_\zeta+R_\zeta\big)$ 

where $\gamma_0=q/\kappa^2$. 
&
Stability in determining $q$ 

(see \cite{BZ}).\\ 
\end{tabular}
\end{center}
\label{default}
\end{table}%
Roughly speaking, when solving the listed inverse problems, we often find ourselves in two scenarios after certain reduction, where CGO solutions turn out to be useful,
\begin{itemize}
\item The first scenario seen in QPAT, ESC and TE (Section 2-4) is when the parameters of interest are associated to the gradient of $u$, where $u$ is the solution to the underlying equation. Simple algebra reduces the problem to solving a transport equation for the unknown parameters. The solvability of the transport equation relies on the density of vector fields $\beta$, which is written in terms of $\nabla u$. Another example is in AET (Section 5) where the internal data itself is a functional of $\nabla u$. The overall strategy here is to use CGO solutions so $\nabla u\sim i\zeta e^{i\zeta\cdot x}$ for $|\zeta|$ sufficiently large. With well-chosen $\zeta$'s, one obtains sufficiently many linearly independent vectors $\beta$ so the parameters can be recovered by solving the transport equations.
\item In the second case, for example for the system model of QTAT (Section 6), (also see \cite{BGM, BM2011Redundant} for anisotropic conductivities in AET and UMEIT, ), the linearized inverse problem is considered. As a result, it is reduced to solving a boundary value problem for a system of sometimes overdetermined (pseudo-) differential equations. By the Douglis-Nirenberg theory, the ellipticity of the boundary value problem provides the stability estimate for the linearized inverse problem. In particular, for the principal symbol of the (pseudo-) differential operator to be non-degenerate, one needs to show again there are sufficient linear independent vector fields. This can be achieved by plugging well-chosen CGO solutions.
\end{itemize}
These are the two commonly seen approaches to apply CGO solutions to inverse problems with internal data. There are other scenarios where CGO solutions can be used on a case-by-case basis suggested by the special structure of underlying equations. For example, CGO is used to give the Fourier transform of the internal data in the first step of AET, and to form a contraction of the unknown parameters in the scalar case of QTAT. \\

We would like to point out that this paper by no means has covered every aspect or method used in tackling inverse problems with similar types of internal information of the solutions. For example, we do not attempt to include the review on reconstruction methods. However, we would like to mention a local reconstruction scheme introduced in \cite{BG, BGM2,BUcpam} where a local linear independence condition needs to be satisfied in order to guarantee the reconstruction. In \cite{BUcpam}, harmonic polynomials are used to show the condition locally. This condition can be extended for global reconstruction by Runge approximation/unique continuation in some situations, e.g., in \cite{BGM3,BUcpam} by applying CGO solutions. Such method is versatile in reconstructing anisotropic tensor-valued parameters.  
Other features of the inverse problems with internal data can also be found in another extensive review \cite{Bal2}.\\


\section{Quantitative Photo-acoustic Tomography} \label{sec:QPAT}
In the present section and the two sections following this, we present methods sharing a general strategy of the CGO application, which was first developed by Bal and Uhlmann in \cite{BU2010}. The idea is to reduce the equations modeling these problems to the Schr\"{o}dinger's equation or Maxwell's equations, and then inserting sufficiently many internal functions $\mathsf{H}$ to obtain a transport equation in one of the unknown parameters. The uniqueness and stability of the recovery of this unknown finally rely on the uniqueness and stability of the solution to the transport equation. Then CGO solutions are used to show the solvability of the transport equation. This idea is initiated by \cite{BU2010} and then further expanded and utilized in the analysis of many coupled physics inverse problems. Here we review some results in photoacoustic tomography \cite{BU2010,CY2012}. The results for electro-seismic conversion \cite{CY2013} is presented in Section \ref{sec:ESC}, and that of transient elastography \cite{Liso} is presented in Section \ref{sec:TE}.


In both PAT and TAT, the first step of the recovery procedure is to reconstruct the absorbed radiation from the boundary measurements of the acoustic waves. This step is typically modelled as an inverse source problem for the acoustic wave equation. This problem has been extensively studied in the mathematical literature, see \cite{Acosta_M, ABJK, FPR, HSBP, HKN, KK, KunyanskiHC_2014, KN2015, PS, SU, St-Y-AveragedTR, SY2016}. We will assume in this section that the absorbed radiation $\mathsf{H}(x)$ has been recovered and concentrate on the second step.

The second step of PAT and TAT is modeled by different equations due to the difference of the radiation used. In PAT, the high frequency radiation (near-infrared laser pulses) is modeled by the diffusion equation, while in TAT the low frequency microwave is modeled by Maxwell's equations. The second step in PAT and TAT are usually referred to as Quantitative Photo-acoustic Tomography (QPAT) and Quantitative Thermo-acoustic Tomography (QTAT), respectively. We consider the QPAT model in this section. A detailed discussion of the QTAT model can be found in Section \ref{sec:QTAT}. 

In QPAT, radiation propagation is modeled by the following diffusion equation
$$\left\{
\begin{array}{rl}
-\nabla\cdot\gamma(x)\nabla u+\sigma(x)u= & 0 \quad\quad\textrm{ in }\Omega \vspace{1ex}\\
u|_{\partial\Omega}= & f.
\end{array}\right.
$$
Here $\gamma(x)$ is the diffusion coefficient, $\sigma(x)$ is the absorption coefficient, $f$ is the illumination on the boundary. The measurement is the absorbed radiation 
$$\mathsf{H}(x) = \sigma(x)u(x),$$
which we assume to be known after solving the first step. 
The objective of QPAT is to reconstruct $(\sigma,\Gamma)$ from the knowledge of $\mathsf{H}(x)$ obtained for a given number of illuminations $f$. For references on QPAT, see e.g., \cite{BR2011, BR2012,BU2010,CAB,CLB,RN,Z}.

\subsection{Full Data} \label{Sec_convert}

In this part we review some full data results on the QPAT model due to Bal and Uhlmann \cite{BU2010}. The QPAT model is described by the diffusion equation with internal data. It was observed in \cite{BU2010} that the inverse problem of the diffusion equation with internal data can be reduced to the one of the Schr\"{o}dinger's equation by the following transform.

Define $v=\sqrt{\gamma}u$, $q=-\frac{\Delta\sqrt\gamma}{\sqrt\gamma}-\frac{\sigma}{\gamma}$ and $\mu=\frac{\sigma}{\sqrt{\gamma}}$, then $v$ satisfies
\begin{equation}
\label{eqn_schr}\Delta v+qv=0 \quad\textrm{ in }\Omega \quad\quad\quad v|_{\partial\Omega}=\sqrt{\gamma}f,
\end{equation}
and the internal measurement becomes $\mathsf{H}(x)=\sigma u=\mu v$. The goal here is to recover $q$ and $\mu$ from $\mathsf{H}(x)$. Then by the definition of $q$, one can solve
$$(\Delta+q)\sqrt{\gamma}=-\mu$$
for $\sqrt{\gamma}$, assuming $\gamma|_{\partial\Omega}$ is known. Finally $\sigma=\mu\sqrt{\gamma}$ recovers $\sigma$.

Therefore, it remains to consider the inverse problem for the Schr\"{o}dinger's equation with internal data. For this purpose, the authors of \cite{BU2010} construct a class of CGO solutions with higher regularity.

\subsubsection{Smoother CGO solutions}

%
%

We begin by reviewing the main ingredients in the construction of $L^2$-CGO solutions initiated by Sylvester and Uhlmann \cite{SU1987}, based on which smoother CGO solutions will be constructed for later applications.

Let $\zeta\in\mathbb{C}^{n}$ be a complex vector with $\zeta\cdot\zeta=0$. Define the space $L^{2}_{\delta}$ $(\delta\in\mathbb{R})$ to be the completion of $C^{\infty}_{c}(\mathbb{R}^{n})$ with respect to the norm
$$\|u\|_{L^{2}_{\delta}}:=\left(\displaystyle\int_{\mathbb{R}^{n}}\langle x\rangle^{2\delta}|u|^{2}\,dx\right)^{\frac{1}{2}}, \quad\quad \langle x\rangle=(1+|x|^{2})^{\frac{1}{2}}.$$
Notice that the function $u_\zeta := e^{i \zeta\cdot x}(1+\psi_{\zeta}(x))$ is a solution of the Schr\"{o}dinger equation $(\Delta+q)u=0$ if and only if $\psi_{\zeta}$ solves
\begin{equation}\label{eqn_correction}
\Delta\psi_{\zeta}+2 i \zeta\cdot\nabla\psi_{\zeta}=-q(1+\psi_{\zeta}) \quad\quad \text{ in } \mathbb{R}^n.
\end{equation}
It remains to find solutions $\psi_\zeta$ of \eqref{eqn_correction}. To this end, denote Faddeev's Green kernel by
\begin{equation} \label{eqn_Faddeev}
G_\zeta(f) := \mathcal{F}^{-1}\left(\frac{\mathcal{F}{f}}{-|\xi|^2 + 2\zeta\cdot\xi}\right),
\end{equation}
where $\mathcal{F}$ is the Fourier transform. It is shown \cite{SU1987} that for $|\zeta|$ large one has
\begin{equation}\label{eqn_homo}\|G_\zeta\|_{L^2_{\delta+1}\rightarrow L^2_\delta}\leq\frac{C}{|\zeta|}\end{equation}
when $-1<\delta<0$. Therefore if $|\zeta|$ is sufficiently large, the equation \eqref{eqn_correction} has a unique solution $\psi_\zeta$ by the fixed point theorem. This is the following result.
\begin{proposition}[\cite{SU1987}]\label{prop:oper-est}
Let $q\in L^\infty(\Omega)$ and $-1<\delta<0$. For any $\zeta\in\mathbb{C}^n$ with $\zeta\cdot\zeta=0$, there exists a unique solution to the Schr\"{o}dinger equation $(\Delta+q)u=0$ of the form
\begin{equation}\label{eqn:CGO}u_{\zeta}(x)=e^{i \zeta\cdot x}(1+\psi_{\zeta}(x))\end{equation}
with $\psi_\zeta\in L^2_\delta$. Moreover, $\psi_\zeta$ satisfies the estimate
$$\|\psi_\zeta\|_{L^2_\delta} \leq \frac{C}{|\zeta|} \left\| q \right\|_{L^2_{\delta+1}}.$$
\end{proposition}


Smoother CGO solutions in higher order Sobolev spaces can be constructed with minor modifications as follows \cite{BRUZ, BU2010}. Introduce the weighted Sobolev space $H^{s}_{\delta}$ $(s\geq 0)$ as the completion of $C^{\infty}_{c}(\mathbb{R}^{n})$ with respect to the norm
$$\|u\|_{H^{s}_{\delta}}:=\left(\displaystyle\int_{\mathbb{R}^{n}}\langle x\rangle^{2\delta}|(I-\Delta)^{\frac{s}{2}}u|^{2}\,dx\right)^{\frac{1}{2}}.$$
Here $(I-\Delta)^{\frac{s}{2}}$ is a pseudodifferential operator whose symbol is $(1+|\xi|^{2})^{\frac{s}{2}}$. Noticing that the two constant coefficient operators $(\Delta+2 i \zeta\cdot\nabla)$ and $(I-\Delta)^{\frac{s}{2}}$ commute, one obtains
\begin{equation}\label{ineq_homo}
\|G_\zeta\|_{H^{s}_{\delta+1}\rightarrow H^s_\delta}\leq\displaystyle\frac{C}{|\zeta|}.
\end{equation}
Finally, by a Neumann series argument, a solution  $\psi_{\zeta}$ to \eqref{eqn_correction} is obtained and satisfies
\begin{equation}\label{ineq_sobolev}
\|\psi_{\zeta}\|_{H^{s}_{\delta}}\leq\displaystyle\frac{C}{|\zeta|}\|q\|_{H^{s}_{\delta+1}}
\end{equation}
when $s=\frac n 2 +k+\epsilon$ for some $k$ positive integer and $\epsilon>0$.
Restricting to the bounded domain $\Omega$ where $q$ is compactly supported, and applying Sobolev embedding theorem yield
\begin{equation}\label{ineq_ck}
\|\psi_{\zeta}\|_{C^{k}(\overline{\Omega})}\leq\displaystyle\frac{C}{|\zeta|}\|q\|_{H^{\frac{n}{2}+k+\epsilon}(\Omega)}.
\end{equation}

For the consideration of the inverse problem,  define the set of admissible parameters
$$
\begin{array}{rl}
\mathcal{P}:= & \Big\{(\mu,q)\in C^{k+1}(\overline{\Omega})\times H^{\frac{n}{2}+k+\epsilon}(\Omega): 0 \textrm{ is not an eigenvalue of } \Delta+q, \vspace{1ex}\\
 & \|\mu\|_{C^{k+1}(\overline{\Omega})}+\|q\|_{H^{\frac{n}{2}+k+\epsilon}(\Omega)}\leq M<\infty,
 \textrm{ and }  \mu \textrm{ is bounded away from } 0 \Big\}.\\
\end{array}
$$

\subsubsection{Uniqueness and stability}

Suppose $\partial\Omega$ is of class $C^{k+1}$, $g_{j}\in C^{k,\alpha}(\partial\Omega;\mathbb{C})$, $j=1,2$, with $\alpha>\frac{1}{2}$, and $(\mu, q)\in \mathcal{P}$. Then the following problem
\begin{equation}\label{eqn_Schrodinger}
(\Delta+q)u_{j}=0 \quad\textrm{ in }\Omega, \quad\quad u_{j}|_{\partial\Omega}=g_{j}
\end{equation}
admits a unique solution $u_{j}\in C^{k+1}(\Omega)$. 
From this we verify that $u_{1}\Delta u_{2}-u_{2}\Delta u_{1}=0$. Taking into consideration $u_{j}=\frac{\mathsf{H}_{j}}{\mu}$, one has
\begin{equation}\label{eqn_transport}
(\mathsf{H}_{1}\nabla \mathsf{H}_{2}-\mathsf{H}_{2}\nabla \mathsf{H}_{1})\nabla\mu+\displaystyle\frac{1}{2}(\mathsf{H}_{2}\Delta \mathsf{H}_{1}-\mathsf{H}_{1}\Delta \mathsf{H}_{2})\mu=0.
\end{equation}
The unique solvability of this transport equation in $\mu$ depends heavily on the behavior of the computable vector field
\begin{equation}\label{eqn:beta}\beta:=\mathsf{H}_{1}\nabla \mathsf{H}_{2}-\mathsf{H}_{2}\nabla \mathsf{H}_{1}.\end{equation}
 By taking $\{u_1, u_2\}$ to be the CGO solutions $\{u_\zeta, u_{\overline\zeta}\}$ in \eqref{eqn:CGO}, where $\zeta=\alpha_{1}+i\alpha_{2}$ with $\alpha_{1},\alpha_{2}\in\mathbb{R}^{n}$, $|\alpha_{1}|=|\alpha_{2}|=\frac{1}{\sqrt{2}}|\zeta|$ and $\alpha_{1}\cdot\alpha_{2}=0$,
one can compute after some basic algebra that
$$\beta=2i\mu^{2}e^{2\alpha_{1}\cdot x}|\alpha_{2}|\left(\displaystyle\frac{\alpha_{2}}{|\alpha_{2}|}+\mathcal{O}\left(\frac{1}{|\zeta|}\right)\right) \quad\quad\textrm{ in }C^{k}(\overline{\Omega}),$$
using the estimate \eqref{eqn_correction}.
This expression indicates that, as $|\zeta|$ becomes larger, the direction of $\Im\beta$ (the imaginary part of $\beta$) becomes more consistent with the vector field $\frac{\alpha_{2}}{|\alpha_{2}|}$. Moreover, it is non-vanishing, implying that each point in $\Omega$ is connected to a point on $\partial\Omega$ by an integral curve of $\Im\beta$. Then the transport equation \eqref{eqn_transport} is uniquely solvable and one can take the imaginary part of it to solve uniquely for $\mu$.

It is easy to see that $\overline{u_\zeta}=u_{\overline\zeta}$. Then the analysis above shows that by taking two real-valued boundary illuminations such as $\{\Re u_\zeta|_{\partial\Omega}, \Im u_\zeta|_{\partial\Omega}\}$ one is able to uniquely determine $\mu$. However, for this to provide a reconstruction scheme, one needs to know the boundary values of the CGO solutions, which are not available. For this to be partly resolved, it is shown in \cite{BU2010} that in fact this set of boundary illuminations can be made larger, that is, close enough to $\{\Re u_\zeta|_{\partial\Omega}, \Im u_\zeta|_{\partial\Omega}\}$, using the elliptic regularity theory.


Converting the above unique determination result back to the diffusion equation case, Bal and Uhlmann obtained in \cite{BU2010} the following theorem. Let $k\geq1$ and set
$$\mathcal{M}:=\{(\gamma,\sigma)\in H^{\frac{n}{2}+k+2+\epsilon}(\Omega)\times C^{k+1}(\overline{\Omega}): \|\sqrt\gamma\|_{H^{\frac{n}{2}+k+2+\epsilon}}+\|\sigma\|_{C^{k+1}(\overline{\Omega})}\leq M<\infty\}.$$
\begin{theorem}[\cite{BU2010}]
Assume that $(\gamma,\sigma)$ and $(\tilde{\gamma},\tilde{\sigma})$ are in $\mathcal{M}$ with $\gamma|_{\partial\Omega}=\tilde{\gamma}|_{\partial\Omega}$. Let $\mathsf{H}=(\mathsf{H}_1, \mathsf{H}_2)$ and $\tilde{\mathsf{H}}=(\tilde{\mathsf{H}}_1, \tilde{\mathsf{H}}_2)$ be the internal data for the coefficients $(\gamma,\sigma)$ and $(\tilde{\gamma},\tilde{\sigma})$, respectively, and with boundary conditions $g:=\{g_1, g_2\}$. Then there is an open set of illuminations $g\in (C^{1,\alpha}(\partial\Omega))^{2}$ for some $\alpha>\frac{1}{2}$ such that if {$\mathsf{H}=\tilde{\mathsf{H}}$}, then $(\gamma,\sigma)=(\tilde{\gamma},\tilde{\sigma})$.
\end{theorem}

By taking a closer look at the behavior of the integral curve of the vector field $\beta$, a Lipschitz type stability is also derived in \cite{BU2010} under certain assumptions on the geometry of $\Omega$ from two real-valued measurements. Since the proof is not a direct application of CGO solutions, we refer interested readers to \cite{BU2010} for more details. \\

With more measurements available, the above idea is further developed in \cite{BU2010} to derive a stability result. More specifically, by imposing on the boundary $2n$ real-valued ($n$ is the spatial dimension) illuminations that are taken to be the real and imaginary parts of $n$ well-chosen CGO solutions of the diffusion equation, we obtain $n$ internal complex measurements $\mathsf{H}^c:=(\mathsf{H}_1^c,\ldots, \mathsf{H}_n^c)$. These internal functions can be used in the same manner as \eqref{eqn:beta} to generate a collection of vector fields $(\beta_1,\ldots,\beta_n)$ whose imaginary parts approximately form a basis of $\R^n$.
As a result, we obtain a system of transport equations, leading to
\begin{equation}\label{eqn_system}
\nabla\mu+\Gamma(x)\mu=0,
\end{equation}
where $\Gamma(x)\in (C^{k}(\overline{\Omega}))^n$ satisfies
\begin{equation}\label{ineq_system}
\|\Gamma-\widetilde{\Gamma}\|_{(C^{k}(\overline{\Omega}))^n}\leq C\|\mathsf{H}^c-\tilde{\mathsf{H}}^c\|_{(C^{k+1}(\overline{\Omega}))^{n}}.
\end{equation}
Then the uniqueness of $\mu$ follows from the uniqueness of the solution of the system \eqref{eqn_system}, and the stability of this reconstruction of $\mu$ follows from the continuous dependence of the solution $\mu$ on the coefficient $\Gamma(x)$ as well as \eqref{ineq_system}. Converting back to the diffusion equation yields
\begin{theorem}[\cite{BU2010}]
Under the above notations. Assume that $(\gamma,\sigma)$ and $(\tilde{\gamma},\tilde{\sigma})$ are in $\mathcal{M}$ with $\gamma|_{\partial\Omega}=\tilde{\gamma}|_{\partial\Omega}$. Let $\mathsf{H}=(\mathsf{H}_{1},\dots,\mathsf{H}_{2n})$ and $\tilde{\mathsf{H}}=(\tilde{\mathsf{H}}_{1},\dots,\tilde{\mathsf{H}}_{2n})$ be the real-valued internal data for the coefficients $(\gamma,\sigma)$ and $(\tilde{\gamma},\tilde{\sigma})$ respectively. Then there is an open set of illuminations $g\in (C^{k,\alpha}(\partial\Omega))^{2n}$ for some $\alpha>\frac{1}{2}$ and a constant $C>0$ such that
$$\|\gamma-\tilde{\gamma}\|_{C^{k}(\overline{\Omega})}+\|\sigma-\tilde{\sigma}\|_{C^{k}(\overline{\Omega})}\leq C\|\mathsf{H}-\tilde{\mathsf{H}}\|_{(C^{k+1}(\overline{\Omega}))^{2n}}.$$
\end{theorem}

\subsection{Partial data}

Under certain circumstances, imposing radiation on the whole $\partial\Omega$ may be either too costly or impossible, thus it is necessary to consider partial data problems. In QPAT, the partial data problem is considered in \cite{CY2012}. The key idea is to apply a special type of CGO solutions that vanish on part of the boundary. These solutions were first constructed by Kenig, Sj\"{o}strand and Uhlmann in \cite{KSU}, which we briefly explain in the following.

To construct such partial data CGO solutions, the linear phase function $\rho\cdot x$ ($\rho=\Re\zeta$) is replaced by a so-called limiting Carleman weight to allow more freedom. A limiting Carleman weight $\varphi$ is a real-valued $C^{\infty}$ function on $\Omega$ such that $\nabla\varphi$ is non-vanishing and the following relation holds:
$$\langle\varphi''\nabla\varphi,\nabla\varphi\rangle+\langle\varphi''\xi,\xi\rangle=0 \quad\textrm{ whenever } |\xi|^2 = |\nabla\varphi|^2\textrm{ and }\nabla\varphi\cdot\xi=0, $$
where $\varphi''$ is the Hessian matrix of $\varphi$. This relation can be viewed as a generalization of the relation $\zeta\cdot\zeta=0$, thus the linear phase $\rho\cdot x$ is a limiting Carleman weight. Another example, which will be used later, is the function $\log|x-x_{0}|$ where $x_{0}\in\mathbb{R}^{n}$ is a fixed point outside of the convex hull of $\Omega$. It also follows from the definition that if $\varphi$ is a limiting Carleman weight, so is $-\varphi$.

Taking $\varphi(x)=\log|x-x_{0}|$ with $x_{0}$ outside of the convex hull of $\Omega$, one divides $\partial\Omega$ into two parts: the front side
$$\partial_{+}\Omega:=\{x\in\partial\Omega:(x_{0}-x)\cdot\nu\geq 0\}$$
and the back side
$$\partial_{-}\Omega:=\{x\in\partial\Omega:(x_{0}-x)\cdot\nu\leq 0\},$$
where $\nu$ is the unit outer normal vector field on $\partial\Omega$. Let $\Gamma$ be a neighborhood of $\partial_{+}\Omega$ in $\partial\Omega$, and $\Gamma_{-}=\partial\Omega\backslash\Gamma$.
\begin{proposition}[\cite{KSU}]\label{prop:CGO_pd}
Let $q\in L^{\infty}(\Omega)$. There exists a real function $\psi\in C^{\infty}(\Omega)$ with $|\nabla\varphi|=|\nabla\psi|$ and $\nabla\varphi\cdot\nabla\psi=0$ such that $(\Delta+q)u=0$ admits solutions of the form
$$u=e^{\frac{1}{h}(\varphi+i\psi)}(a+r)+z, \quad\quad\quad u|_{\Gamma_{-}}=0.$$
Here $a\in C^{\infty}(\overline{\Omega})$ is non-vanishing, $h>0$ is a small parameter, $z=e^{i\frac{l}{h}}b(x;h)$ with $b$ a symbol of order $0$ in $h$ and
$$\Im l(x)=-\varphi(x)+k(x)$$
where $k(x)\sim$ dist$(x,\partial\Omega_{-})$ in a neighborhood of $\partial\Omega_{-}$ and $b$ has its support in that neighborhood. Moreover, $\|r\|_{L^{2}(\Omega)}=\mathcal{O}(h)$, $r|_{\partial\Omega_{-}}=0$, $\|(\nabla\varphi\cdot\nu)^{-\frac{1}{2}}r\|_{L^{2}(\partial\Omega_{+})}=\mathcal{O}(h^{\frac{1}{2}})$.
\end{proposition}
One may also increase the regularity to obtain smoother CGO solutions \cite{CY2012}. Then the ideas in the previous full data section can be adapted to obtain uniqueness and stability for partial data problems. Let $(\mu,q)\in\mathcal{P}$. For any two solutions of $(\Delta+q)u_{j}=0$ ($j=1,2$), \eqref{eqn_transport} is still valid. Instead of \eqref{eqn:CGO}, we plug in the CGO solutions from Proposition \ref{prop:CGO_pd}
$$u_{1}=e^{\frac{1}{h}(-\varphi+i\psi)}(a_{1}+r_{1})+z_{1}, \quad\quad u_{2}=e^{\frac{1}{h}(\varphi+i\psi)}(a_{2}+r_{2})+z_{2}.$$
Then the vector field $\beta$ satisfies
$$\left\|\beta-\displaystyle\frac{2}{h}a_{1}a_{2}\mu^{2}e^{\frac{2i}{h}\psi}\frac{x_{0}-x}{|x_{0}-x|^{2}}\right\|_{C^{k}(\overline{\Omega})}=\mathcal{O}(h) \quad\textrm{ as } h\rightarrow 0.$$
Notice that the vector $\frac{x_{0}-x}{|x_{0}-x|^{2}}$ points to $x_{0}$, hence the integral curves of $\beta$ hit $\partial\Omega$ near the front side $\partial_{+}\Omega$. This along with the fact that $u_j|_{\Gamma_-}=0$ allow to make measurements only near $\partial_{+}\Omega$ to have a unique solution $\mu$ for the transport equation. Notice that to assemble the two CGO solutions, one needs 4 real-valued boundary illuminations. Therefore, one has
\begin{theorem}[\cite{CY2012}]
With the above notations, suppose $(\gamma,\sigma)$ and $(\tilde{\gamma},\tilde{\sigma})$ are in $\mathcal{M}$ and $\gamma|_{\Gamma}=\tilde{\gamma}|_{\Gamma}$. Then there is an open set of illuminations $g:=(g_1,g_2,g_3,g_4)\in(C^{1,\alpha}(\partial\Omega))^{4}$ with supp $g\subset\Gamma$ for some $\alpha >\frac{1}{2}$ such that if the corresponding internal measurements satisfy $\mathsf{H}_j=\tilde{\mathsf{H}}_j$ ($j=1,\ldots,4$), then $(\gamma,\sigma)=(\tilde{\gamma},\tilde{\sigma})$ in $\overline{\Omega}$.
\end{theorem}

Similar to the sequence of results in \cite{BU2010}, with more measurements we have

\begin{theorem}[\cite{CY2012}]
With the above notations, 
suppose $(\gamma,\sigma)$ and $(\tilde{\gamma},\tilde{\sigma})$ are in $\mathcal{M}$ and $\gamma|_{\Gamma}=\tilde{\gamma}|_{\Gamma}$. Then there is an open set of illuminations $g\in(C^{k,\alpha}(\partial\Omega))^{4n}$ with supp $g\subset\Gamma$ for some $\alpha >\frac{1}{2}$ and a constant $C$ such that in $\Omega$
$$\|\gamma-\tilde{\gamma}\|_{C^{k}(\Omega)}+\|\sigma-\tilde{\sigma}\|_{C^{k}(\Omega)}\leq C\|\mathsf{H}-\tilde{\mathsf{H}}\|_{(C^{k+1}(\Omega))^{4n}},$$
where $\mathsf{H}=(\mathsf{H}_{1},\cdots,\mathsf{H}_{4n})$ and $\tilde{\mathsf{H}}=(\tilde{\mathsf{H}}_{1},\cdots,\tilde{\mathsf{H}}_{4n})$ are the corresponding internal data.
\end{theorem}

\section{Electro-seismic Conversion}\label{sec:ESC}

Seismo-Electric (SE) and Electro-Seismic (ES) conversions are phenomenon occurred in fluid-saturated porous media. These conversions couple electromagnetic waves and elastic waves through the electro-kinetic effect. SE conversion employs seismic source to generate electromagnetic waves, while ES conversion emits electromagnetic waves to excite elastic waves. These conversions have been applied in oil prospection as well as other geophysics studies. Detailed description of the physical mechanism underlying these conversions can be found, for instance, in \cite{CY2013}. Theoretical and experimental results on SE conversion have been obtained in, e.g. \cite{BRKM1996, MHT1997, MQT2000, TG1993, T1936}. However we concentrate on ES conversion in this section

The governing equations of ES conversion are derived by Pride \cite{P} based on Biot's theory on the elastic wave propagation in porous media \cite{B1956, B1956_2}. Pride's equations are analyzed in \cite{SZG2012, W2005, WZ2006} and tested in \cite{HT2007, PH1996, T2005, THBM2005}. The coupled-physics inverse problem to be considered below is the inverse problem of the linearized electro-seismic conversion. It consists of two steps: the first step is to inverse Biot's equations \cite{P} to recover the internal potential from boundary measurements, and the second step is the inversion of Maxwell's equation with internal measurement. The first step has been investigated in \cite{P} and an approximation method is proposed there. Here we review the results in \cite{CY2013} where it is assumed that the first step has been successfully implemented and investigates only the second step.

Let $\Omega\subset\mathbb{R}^{3}$ be an open, bounded and connected domain with $C^2$ boundary $\partial\Omega$. 
The propagation of the electric fields in ES conversion is modeled by the following Maxwell's equations when there is no source current.
\begin{equation}\left\{
\begin{array}{rcl}\label{eqn_Maxwell}
        \nabla\times E &=& i\omega\mu_{0} H,\\
        \nabla\times H &=& (\sigma - i\varepsilon\omega)E.
\end{array}
\right.\end{equation}
Here $\omega>0$ denotes a fixed seismic  
wave frequency, the constant $\mu_{0}>0$ denotes the magnetic permeability, $\sigma=\sigma(x)$ the conductivity, $\varepsilon=\varepsilon(x)$ the relative permittivity, $E$ the electric field and $H$ the magnetic field. In this section it is assumed that $\varepsilon(x)=\varepsilon_{0}$ and $\sigma(x)=0$ for all $x$ outside some sufficiently large ball containing $\Omega$, where $\varepsilon_{0}>0$ is the dielectric constant. The measurement, which is obtained from the first step, is the internal potential given by
$$\mathsf{H}:=LE\quad\quad\textrm{ in }\Omega,$$
where $L=L(x)$ is the coupling coefficient.
The controllable boundary illumination is the tangential boundary electric field
$$G=tE:=\nu\times E\quad\quad\textrm{ on }\partial\Omega ,$$
where $\nu$ denotes the unit outer normal on $\partial\Omega$. The question is similar to the one in QPAT: by choosing boundary illuminations $G$, can we uniquely determine the pair $(L,\sigma)$ from the internal data $\mathsf{H}$? For small $\iota>0$, it will be shown that this is possible for the set of coefficients
$$
\begin{array}{rl}
\mathcal{M}:= & \{(L,\sigma)\in C^{k+1}(\overline{\Omega})\times H^{\frac{3}{2}+3+{k}+\iota}(\Omega): 0 \textrm{ is not an eigenvalue of } \vspace{1ex}\\
 &  \nabla\times\nabla\times\cdot -\kappa^2n \}, \\
\end{array}
$$ 
where the wave number $\kappa>0$ and the refractive index $n(x)$ are given by
$$\kappa=\omega\sqrt{\varepsilon_0\mu_0}, \quad\quad n= \varepsilon_{0}^{-1}\left(\varepsilon+i\omega^{-1}\sigma \right).$$

The authors of \cite{CY2013} proved the uniqueness and stability to the reconstruction of the pair $(L,\sigma)$ under appropriate assumptions. The key ingredient is the construction of special CGO solutions. With these solutions it is possible to derive again a transport equation in $L$, from which the uniqueness and stability results follow. The solutions used in \cite{CY2013} take the form
\begin{equation}\label{CGOs}
        E(x) = e^{i\zeta\cdot x}(\eta + R_\zeta(x)),
\end{equation}
where $\zeta\in\mathbb{C}^3\backslash\mathbb{R}^3$, $\eta\in\mathbb{C}^3$, are constant vectors satisfying
\begin{equation}
        \label{eqn_CGO_para_con}
        \zeta\cdot\zeta = \kappa^2, \;
        \zeta\cdot\eta = 0.
\end{equation}
$C^{2}(\overline{\Omega})$-Solutions of this form to the Maxwell's equations are introduced by Colton and P\"{a}iv\"{a}rinta \cite{CP}. In \cite{CY2013} $C^{k+1}(\overline{\Omega})$-solutions of this type are constructed by studying the property of Faddeev kernel. Write $\alpha=\nabla n(x)/n(x)$, the asymptotic behavior of these solutions is given by the next proposition where $G_\zeta$ is the Faddeev's Green kernel defined in \eqref{eqn_Faddeev}.
\begin{proposition}[\cite{CY2013}] \label{thm:CGO2}
Suppose $\zeta\in\mathbb{C}^3\backslash\mathbb{R}^3$ and $\eta\in\mathbb{R}^3$, satisfying $\zeta\cdot\zeta = \kappa^2$ and $\zeta\cdot\eta = 0$ such that as $|\zeta|\rightarrow \infty$ the limits $\zeta/|\zeta|$ and $\eta$ exist and,
$$|\zeta/|\zeta| - \zeta_0| = O\left( |\zeta|^{-1}\right),\quad |\eta - \eta_0| = O\left(|\zeta|^{-1}\right).$$
Then there exist solutions to \eqref{eqn_Maxwell} of the form \eqref{CGOs} with $R_{\zeta}\in C^{k+1}(\overline{\Omega})$. Moreover,
$$\|R_\zeta - i|\zeta|n^{-1/2}G_\zeta(n^{1/2}\alpha\cdot \eta_0) \zeta_0\|_{C^{{k}+1}(\overline{\Omega})} = O\left(|\zeta|^{-1}\right)$$ 
as $|\zeta|\rightarrow \infty$.
\end{proposition}

Specific vectors $\zeta_{j}$ and $\eta_{j}$, $j=1,2$, are chosen in \cite{CY2013} so that 
\begin{equation}\label{zetaeta}
\begin{array}{rcl}
        \displaystyle\lim_{|\zeta|\rightarrow \infty} \eta_j = \eta_0 &:=& (1,0,0), \quad j = 1,2,\\[6pt]
        \displaystyle\lim_{|\zeta|\rightarrow \infty} \zeta_1/|\zeta_1| &=&\zeta_0 := \frac{1}{\sqrt{2}}(0,i,1),\\[6pt]
        \displaystyle\lim_{|\zeta|\rightarrow \infty} \zeta_2/|\zeta_2| &=& -\zeta_0. 
\end{array}
\end{equation}
Proposition \ref{thm:CGO2} implies that
\begin{equation}\label{eqn_asymptotic}
(\eta_1 + R_{\zeta_1})\cdot(\eta_2 + R_{\zeta_2}) = 1 + o_{|\zeta|\rightarrow \infty}(1)
\end{equation}
in $C^{{k}+1}(\overline{\Omega})$. 

To reconstruction $(L,\sigma)$, first notice that from \eqref{eqn_Maxwell} one has
$$\nabla\times\nabla\times E-\kappa^{2}nE=0 \quad\quad\textrm{ in }\Omega.$$
Given any two (possibly complex-valued) solutions $E_{1}$, $E_{2}$ whose tangential boundary illuminations are $tE_{j}=G_{j}$ on $\partial\Omega$ and internal data are $\mathsf{H}_{j}=LE_{j}$, the next identity is valid:
$$\nabla\times\nabla\times E_{1}\cdot E_{2}-\nabla\times\nabla\times E_{2}\cdot E_{1}=0.$$
Replacing $E_{j}$ by $\frac{\mathsf{H}_{j}}{L}$ yields the following transport equation
$$\beta\cdot\nabla L + \gamma L = 0,$$
where
$$
\begin{array}{rcl}
\beta &=& \Big\{ [(\nabla \mathsf{H}_1)\mathsf{H}_2 - (\nabla \mathsf{H}_2)\mathsf{H}_1] + [(\nabla\cdot \mathsf{H}_1)\mathsf{H}_2 - (\nabla\cdot \mathsf{H}_2)\mathsf{H}_1] \\
        && - 2[(\nabla \mathsf{H}_1)^T\mathsf{H}_2 - (\nabla \mathsf{H}_2)\mathsf{H}_1]\Big\},  \\
\gamma &=& \Big\{[\nabla(\nabla\cdot \mathsf{H}_1)\cdot \mathsf{H}_2 - \nabla(\nabla\cdot \mathsf{H}_2)\cdot \mathsf{H}_1] - [\nabla^2\mathsf{H}_1\cdot \mathsf{H}_2 - \nabla^2\mathsf{H}_2\cdot \mathsf{H}_1]\Big\}.
\end{array}
$$
Again what matters here is the behavior of the vector field $\beta$.
Let $E_{1}$, $E_{2}$ be CGO solutions of the form \eqref{CGOs} with $\zeta_j$ chosen by \eqref{zetaeta}.
By \eqref{eqn_asymptotic}, one has
$$\|\beta-iL^{2}\zeta_{0}\|_{C^{k}(\overline{\Omega})}\rightarrow 0 \quad\textrm{ as } |\zeta|\rightarrow \infty.$$
Thus $\beta$ has approximately constant directions for small $h$, and their integral curves connect every internal point to $\partial\Omega$. This ensures that the above transport equation admits a unique solution. By a boundary perturbation argument based on the regularity theory of Maxwell's equations, it is obtained in \cite{CY2013} that
\begin{theorem}[\cite{CY2013}]
Let $(L,\sigma)$ and $(\tilde{L}, \tilde{\sigma})$ be two pairs in $\mathcal{M}$ with $L|_{\partial\Omega}=\tilde{L}|_{\partial\Omega}$. Let $\mathsf{H}:=(\mathsf{H}_1,\mathsf{H}_2)$ and $\tilde{\mathsf{H}}:=(\tilde{\mathsf{H}}_1,\tilde{\mathsf{H}}_2)$, be two sets of internal data on $\Omega$ for the coefficients $(L,\sigma)$ and $(\tilde{L},\tilde{\sigma})$, respectively, corresponding to the boundary (complex) illuminations $G:=(G_1,G_2)$. Then there is a non-empty open set of $G \in (C^{k+4}(\partial\Omega))^2$ such that if $\mathsf{H}=\tilde{\mathsf{H}}$, then $(L,\sigma) = (\tilde{L},\tilde{\sigma})$.
\end{theorem}
Based on the above procedure, stability results can be established as well, see \cite{CY2013} for more details.

The results reviewed in this section are obtained under the assumption that the magnetic permeability $\mu_{0}$ is constant. Recently in \cite{CH} this assumption has been relaxed to functions by converting the system \eqref{eqn_Maxwell} to a matrix Schr\"{o}dinger's equation and utilizing the CGO solutions constructed in \cite{OS}, interested readers are referred to \cite{CH} for more details.

\section{Transient Elastography}\label{sec:TE}
In this section, we consider a hybrid inverse problem involving elastic measurements called Transient Elastography (TE) which enables detection and characterization of tissue abnormalities.

TE is a non-invasive tool for measuring liver stiffness. The device creases high-resolution shear
stiffness images of human tissue for diagnostic purposes. Shear
stiffness is targeted because shear wave speed is larger in abnormal
tissue than in normal tissue. In the experiment tissue initially is
excited with pulse at the boundary. This pulse creates the shear wave
passing through
the liver tissue.
Then the tissue displacement is measured using the ultrasound. The displacement is related to the tissue stiffness because the soft tissue has larger deformation than stiff tissue. When we have
tissue displacement, we want to reconstruct shear modulus $\mu$ and the parameter $\lambda$.
See \cite{MZM} and references there for more details.

This modality, TE, also includes two steps. The first step is to solve an inverse scattering problem for a time-dependent wave equation of ultrasound, which is a high resolution step as PAT and TAT.
The second step is to recover the Lam\'{e} parameters from the knowledge of the tissue displacement obtained from the first step.
In this section, we will again focus on the second step which is to quantitatively reconstruct coefficients that display high contrast from the internal data of tissue displacement.

The mathematical problem is formulated as follows.
Let $\Omega\subset \mathbb{R}^n,\ n=2,3,$ be an open bounded domain with smooth boundary. Let $u=(u_1,\ldots,u_n)^T$ be the displacement satisfying the linear isotropic elasticity system
\begin{equation}\label{origin}
\left\{\begin{array}{rl}
     \nabla\cdot (\lambda(\nabla\cdot u)I+2S(\nabla u)\mu)+k^2 u=0\ \ &\hbox{in $\Omega$},\\
     u=g\ \ &\hbox{on $\partial\Omega$},
     \end{array}
     \right.
\end{equation}
where $S(A)=(A+A^T)/2$ denotes the symmetric part of the matrix $A$. Here $(\lambda,\mu)$ are Lam\'{e} parameters and $k\in \mathbb{R}$ is the frequency. For the forward problem of elasticity system, we refer the readers to \cite{C}. The time-harmonic scalar model for TE was showed by Bal and Uhlmann \cite{BUcpam} where the reconstruction of coefficients in scalar second-order elliptic equations was also studied. The time-harmonic Lam\'e system from internal measurements was also considered in \cite{BBIM}. The reconstruction algorithms for fully anisotropic elasticity tensors from knowledge of a finite number of internal data were derived in \cite{BMU}.

In this section, we concern the linear isotropic elasticity setting. Assume that $k$ is not an eigenvalue of the elasticity system. We impose $J\in\Z^+$ boundary displacements $g^{(j)}$ $(1\leq j\leq J)$, then the set of internal functions obtained by the first step in TE is given by
\[
    \mathsf{H}(x)=(u^{(j)}(x))_{1\leq j\leq J}\ \ \ \hbox{in $\Omega$},
\]
where $u^{(j)}$ denotes the solution to \eqref{origin} with boundary condition $u^{(j)}|_{\partial\Omega}=g^{(j)}$. Notice that the internal functions for QPAT, Electro-seismic and TE are all linear functionals of solutions to underlying PDE's.
In order to recover the Lam\'e parameters from the internal functions $\mathsf{H}$, the strategy is similar: we reduce the inverse problem to solving a transport equation for $\mu$ or $\lambda$. 

To be more specific, since $\Omega$ is bounded, we pick a ball $B_R$ for $R>0$, such that $\overline{\Omega}\subset B_R$ and extend $\lambda$ and $\mu$ to $\mathbb{R}^n$ by preserving their smoothness and also $\textrm{supp}(\lambda)$, $\textrm{supp}(\mu)\subset B_R$.
We will use the reduced system derived by Ikehata \cite{ike}. This reduction was also mentioned in \cite{gbook}. Then the elasticity equations (\ref{origin}) can be transformed to the following equations. 
Suppose $\left(
   w,
   f
\right)^T$ satisfy
\begin{equation}\label{wfa}
    \Delta \left(\begin{array}{c}
   w \\
   f
\end{array}
\right)+V_1(x)\left(\begin{array}{c}
   \nabla f \\
   \nabla\cdot w
\end{array}
\right)+V_0(x)\left(\begin{array}{c}
   w \\
   f
\end{array}
\right)=0.
\end{equation}
Here $V_0$ contains the third derivative of $\mu$ and
$$
     V_1(x)=\left(
            \begin{array}{cc}
              -2\mu^{1/2}\nabla^2\mu^{-1}+\mu^{-3/2}k^2 & -\nabla \log\mu \\
              0 & \displaystyle\frac{\lambda+\mu}{\lambda+2\mu}\mu^{1/2} \\
            \end{array}
          \right).
$$
The solution to the elasticity system (\ref{origin}) is then given by
\begin{equation}\label{cgoelas}
 u:=\mu^{-1/2}w+\mu^{-1}\nabla f-f\nabla\mu^{-1}.
\end{equation}
Here $\nabla^2 g$ denotes the Hessian matrix $(\partial^2g/\partial x_i\partial x_j)_{ij}$.
The CGO solutions to (\ref{wfa}) are presented in the following lemma.
\begin{lemma}[\cite{ER}, \cite{NU1}]\label{Eskin}
{Let $\alpha$ and $\beta$ be two orthogonal unit vectors in $\mathbb{R}^n$. Denote $\zeta=\tau(\alpha+i\beta)$ with $\tau>0$ and $\theta=\alpha+i\beta$. }Consider the Schr$\ddot{o}$dinger equation with external Yang-Mills potentials
\begin{equation}\label{Lu}
     -\Delta U-2iA(x)\cdot \nabla U+B(x)U=0,\ \ \ \ \ x\in B_R\subset\mathbb{R}^n,
\end{equation}
where $A(x)=(A_1(x), \ldots, A_n(x))\in C^{l}(B_R),\ l\geq 6$ with $A_j(x)$ and $B(x)$ are $(n+1)\times (n+1)$ matrices.
Then for $\tau$ sufficiently large, there exists a solution to \eqref{Lu} of the form
\[
     U_\zeta=e^{i\zeta\cdot x} (C_0(x,\theta)p(\theta\cdot x)
                        +O(\tau^{-1})
                        ),
\]
where $C_0\in C^{l}(B_R)$ is the solution to
\[      i\theta\cdot \frac{\partial}{\partial x}C_0(x,\theta)=\theta \cdot A(x)C_0(x,\theta)
\]
and $\det C_0$ is never zero, $p(z)$ is an arbitrary polynomial in complex variables $z$, and $O(\tau^{-1})$ is bounded by $C(1/\tau)$ in $H^s(B_R)$-norm for $0\leq s \leq l -2$.
\end{lemma}
The above lemma provides the CGO solutions $(w_\zeta,f_\zeta)^T$ to \eqref{wfa}.
Substituting it into \eqref{cgoelas}, we obtain the CGO solutions $u_\zeta$ to (\ref{origin}).\\

Meanwhile, let us consider the two dimensional case and the reconstruction of $\mu$ here (the three dimensional case and the reconstruction of $\lambda$ are similar), in which case we can deduce the following equation from (\ref{origin}),
\begin{equation}\label{equ11}
     u^\sharp \cdot F+u^\flat\cdot G=-k^2u^*,
\end{equation}
where $u$ is a solution to (\ref{origin}) and
\[\label{sharp}
     u^\sharp =\left(\begin{array}{c}
         \partial_1(\nabla\cdot u) \\
         \partial_2(\nabla\cdot u) \\
         \nabla\cdot u \\
         \nabla\cdot u \\
       \end{array}
     \right),\ \
     F= \left(
       \begin{array}{c}
         \lambda+\mu \\
         \lambda+\mu \\
         \partial_1 (\lambda+\mu) \\
         \partial_2 (\lambda+\mu) \\
       \end{array}
     \right),
\]
\[
     u^\flat= \left(
       \begin{array}{c}
                a+b \\
                a-b \\
         \partial_1(a+b) \\
         \partial_2(a-b) \\
       \end{array}
     \right),\ \
      G= \left(
       \begin{array}{c}
        \partial_1\mu \\
        \partial_2\mu \\
         \mu \\
         \mu \\
       \end{array}
     \right)
\]
with $a=\partial_2 u_1+\partial_1 u_2,\ b=\partial_1 u_1-\partial_2
u_2$, and $u^*=(u_1+u_2)$.
The component of $u^\sharp$ and $u^\flat$
can be computed from the known internal data
$u$.

Obverse that the vector $u^\sharp$ has three different entries. The rest of the proof is to choose enough proper CGO solutions $u_\zeta$ on some subdomain of $\Omega$ such that
the first term $u^\sharp\cdot F$ in \eqref{equ11} vanishes. Then we obtain a transport equation only about $\mu$ of the form
\begin{equation} \label{eqn:transport_standard}
    \nabla \mu+\Gamma(x) \mu=\Psi(x),
\end{equation}
where $\Gamma$ and $\Psi$ are vector-valued functions that can be computed from the known internal data. The uniqueness and stability of $\mu$ follow immediately from the similar argument as before (also see \cite{BU2010}). 

More specifically, in \cite{Liso}, we choose three suitable CGO solutions $u^{(j)}$ for $j=0,1,2$ to \eqref{origin}.
Now for $\natural=\sharp,\flat, *$, set
$$
\mathfrak{u}_1^{(0)\natural}=\Re\left(\chi_0(x)u^{(0)\natural}\right),\ \ \mathfrak{u}_2^{(0)\natural}=\Im \left(\chi_0(x)u^{(0)\natural}\right),
$$
$$
   \mathfrak{u}^{(1)\natural}=\Re\left(\chi_3(x)u^{(1)\natural}\right),\ \ \mathfrak{u}^{(2)\natural}=\Im\left(\chi_3(x)u^{(1)\natural}\right),
$$
and
$$
   \mathfrak{u}^{(3)\natural}=\Re\left(\chi_1(x)u^{(1)\natural}+\chi_2(x)u^{(2)\natural}\right),
$$
where $\chi_j(x)$ is some nonzero function (see \cite{Liso}). Since $\lambda,\ \mu$ are real-valued functions, we have
\begin{equation}\label{equ123}
       \mathfrak{u}_l^{(0)\sharp}\cdot F+ \mathfrak{u}_l^{(0)\flat}\cdot G=-k^2\mathfrak{u}_l^{(0)*},\ \ l=1,2
\end{equation}
and 
\begin{equation}\label{equ13}
       \mathfrak{u}^{(j)\sharp}\cdot F+ \mathfrak{u}^{(j)\flat}\cdot G=-k^2\mathfrak{u}^{(j)*} \ \ \hbox{for $j=1,2,3$.}
\end{equation}
Based on the choice of CGO solutions in \cite{Liso}, {it is shown that $\left\{\mathfrak{u}^{(1)\sharp},
\mathfrak{u}^{(2)\sharp}, \mathfrak{u}^{(3)\sharp}\right\}$ are three
linearly independent vectors at every $x$ in some subdomain $\Omega_0$ of $\Omega$. Then, for each $l=1,2$} 
there exist three functions
$\Theta^l_1, \Theta^l_2$, and $\Theta^l_3$ such that
$$\mathfrak{u}_l^{(0)\sharp}+\sum^3_{j=1}\Theta^l_j
\mathfrak{u}^{(j)\sharp}=0.$$ Then for $l=1,2$, multiplying (\ref{equ13})
by $\Theta^l_j$ and summing over $j$ and equation (\ref{equ123}), we
have
\begin{equation*}
       \mathrm{v}_l\cdot G=-k^2\left( \mathfrak{u}_l^{(0)*}+\sum^3_{j=1}\Theta^l_j  \mathfrak{u}^{(j)*}\right),
\end{equation*}
{which can be further rewritten in the form }
$$
    \beta_l\cdot \nabla\mu+\gamma_l \mu=-k^2\left( \mathfrak{u}_l^{(0)*}+\sum^3_{j=1}\Theta^l_j  \mathfrak{u}^{(j)*}\right).
$$
{Here again the choice of CGO solutions guarantees that $\beta_1(x)$ and $\beta_2(x)$ are linearly independent for every $x\in \Omega_0$, which implies \eqref{eqn:transport_standard}. Therefore, $\mu$ can be uniquely and stably reconstructed in $\Omega_0$. }

Let us denote $$\mathcal{P}=\{(\lambda,\mu)\in C^{7}(\overline{\Omega})\times C^9(\overline{\Omega});\ 0<m\leq \|\lambda\|_{C^{7}(\overline{\Omega})}, \|\mu\|_{C^{9}(\overline{\Omega})}\leq M\ \  \hbox{and}\ \   \lambda,\ \mu>0 \}.$$
It is shown in \cite{Liso} that the local reconstruction of $\lambda$ in two dimensional case requires additional four CGO solutions, hence one needs seven CGO solutions to obtain the reconstruction of $\lambda$ and $\mu$ together.
The main result is stated in the following theorem.
\begin{theorem}[\cite{Liso}]
Let $\Omega$ be an open bounded domain of $\mathbb{R}^n$ with smooth boundary. Suppose that the Lam\'{e} parameters $(\lambda, \mu)$ and $(\tilde\lambda,\tilde\mu)\in \mathcal{P}$ satisfy $\mu|_{\partial\Omega}=\tilde\mu|_{\partial\Omega}$.
Let $u^{(j)}$ and $\tilde u^{(j)}$ be the solutions of the elasticity system with boundary data $g^{(j)}$ for parameters $(\lambda, \mu)$ and $(\tilde\lambda,\tilde\mu)$, respectively, and $\mathsf{H}=(u^{(j)})_{1\leq j\leq J}$, $\tilde{\mathsf{H}}=(\tilde u^{(j)})_{1\leq j\leq J}$ be the corresponding internal data for some integer $J\geq 3n+1$ . 

Then there is an open set of the boundary data $(g^{(j)})_{1\leq j\leq J}$ such that
$\mathsf{H}=\tilde{\mathsf{H}} $
implies $(\lambda,\mu)=(\tilde \lambda,\tilde\mu)$ in $\Omega$.\\
Moreover, we have the stability estimate
$$
     \|\mu-\tilde\mu\|_{C(\Omega)}+\|\lambda-\tilde\lambda\|_{C(\Omega)}\leq C\|\mathsf{H}-\tilde{\mathsf{H}} \|_{C^2(\Omega)}.
$$
\end{theorem}

\section{Acousto-electric Tomography}\label{sec:AET}
In Acousto-Electric Tomography (AET), also known as the ultrasound modulated electrical impedance tomography (UMEIT), acoustic waves sent to a conductive object cause a slight change in the conductivity of the object. This interaction between the electric and acoustic waves is called  the \textit{acousto-electric effect}.  Even though the change in conductivity due to  acousto-electric effect is small, it can be observed by making electrical boundary measurements {\cite{ZW2004}}.  This observable difference in the
boundary measurement readings is the data for the inverse problem of AET, and the goal is to recover the background
internal conductivity from these measurements {\cite{Ammari2008Elastic, Bal2011, BBMT2011, BM2011Redundant, ImagingByM2009, GS2009, ilker2012, Kuchment10syntheticfocusing, KuKu2010, ZW2004}},

Similar to the other coupled physics methods a two-step approach is used in AET. First step is the reconstruction of an internal data from the actual boundary measurements. More precisely, before the ultrasound modulation, voltage-to-current (or alternatively current-to-voltage) measurements are done
on the boundary as in the case of EIT. Then the same measurements are done while the
object is being scanned by ultrasound waves. 
The difference in these two current responses  is due to the change in the conductivity and  it is used to obtain  internal information about the medium.
The second step is then to obtain the background conductivity from this internal data obtained in the previous step.
Therefore the model here is different from that of EIT, and it is shown that the inverse problem of AET has better stability compared to EIT. 

The change in the conductivity map due to the ultrasound waves is modelled by
\[ {\gamma_m (x) = (1 + m (x)) \gamma (x)}, \]
where $\gamma(x)$ is the background conductivity map, and $m(x)$ is the modulation which is assumed to be a known smooth map that depends on the acoustic signal strength \cite{Ammari2008Elastic, Bal2011, KuKu2010}.
\\


Let $\Omega$ be an open bounded connected domain in $\R^n$ with
smooth boundary $\partial \Omega$ and $\gamma$ be the unknown isotropic background conductivity map which is assumed to be known on the boundary $\partial \Omega$. The conductivity
equation is given by
\begin{equation}
  \nabla \cdot (\gamma \nabla u) = 0 \quad\textrm{in } \Omega, \hspace{1em} u
  |_{\partial \Omega}  = f, \label{eq:cond}
\end{equation}
for a given boundary potential $f \in H^{1 /
2} (\partial \Omega) .$ The Dirichlet-to-Neumann (DtN) map $\Lambda_\gamma: f\mapsto\gamma \partial_\nu u|_{\partial \Omega}$, also known as the voltage-to-current map, can then be defined
by the following quadratic form
\[ \langle \Lambda_\gamma f, g \rangle = \int_{\Omega} \gamma (x) \nabla u
   (x) \nabla u_g (x) \mathrm{d} x, \hspace{1em} g \in H^{1 / 2}
   (\partial \Omega), \]
where $u_g$ is an extension of $g$ to $\Omega$. Then the difference of DtN maps due to ultrasound modulation $m$, evaluated at boundary potential $f$, is
\[ M_f(m) := (\Lambda_{\gamma_m} - \Lambda_{\gamma}) f. \]
The goal of AET is to recover $\gamma (x)$ from $M_f$.

First,  we consider the CGO solutions
 \[u_\zeta={\gamma}^{-1/2}e^{i\zeta \cdot x} (1 + \psi_{\zeta} (x))\]
of the  conductivity equation $\nabla\cdot\gamma\nabla u=0$,
where $\zeta\in\C^n$ satisfies $\zeta\cdot\zeta=0$ and 
 $\psi_{\zeta}\in{H^{s}}$ satisfies (\ref{ineq_sobolev}). The existence and construction of these solutions are given in Section \ref{sec:QPAT}. 
The conductivity equation can be reduced to the Schr\"{o}dinger's equation (\ref{eqn_schr}) using the same transform described in Section \ref{Sec_convert} with   $q=-\frac{\Delta\sqrt\gamma}{\sqrt\gamma}$.
  Then it can be seen that the gradient of these special solutions are of the form
\[ \nabla u_\zeta = e^{i\zeta \cdot x} (\gamma^{-1/2}(1 + \psi_{\zeta}) \zeta + (1
   + \psi_{\zeta}) \nabla(\gamma^{-1/2}) +\gamma^{-1/2} \nabla \psi_{\zeta}). \]
Then, we set
\[ m (x) = e^{(- i\zeta - i  k) \cdot x},  \]
where $ k \in \R^n$ is fixed. Here the map
$M_f$ is extended to complex valued functions.
Let $u, u_m \in H^1 (\Omega)$ be the
solutions to (\ref{eq:cond}) for the conductivities $\gamma$ and $\gamma_m$,
respectively.
We obtain that for $g=u_\zeta|_{\partial\Omega}$,
\begin{equation}
  \begin{array}{ll}
    \langle M_f (m), g \rangle & = \displaystyle\int_{\Omega} e^{- i  k \cdot x}
    \sqrt\gamma\nabla u\cdot \zeta ~\mathrm{d} x + r (g),
  \end{array} \label{eq:mf-approx2}
\end{equation}
where $| r (g) |$ is bounded from above and the bound is independent of $\zeta$. Then it is possible to recover an internal data of the form
$\sqrt\gamma\nabla u$ in $\Omega$. This roughly follows from the
fact that $\sqrt\gamma\nabla u_\zeta$ is almost flat up to a known function thanks to
the behavior of CGO functions that   $\|\psi_{\zeta}\|_{H^{s}}$ decays  as $ |\zeta|$ increases, as mentioned above.

\begin{proposition}{\cite{ilker2012}}
  \label{thm2} Given
  ${M_f}$ for some fixed {{$f \in H^{1 / 2} (\partial \Omega)$}}, define \[ s ( k, \zeta) := \frac{1}{| \zeta |} \langle M_f (m), g \rangle ,\]
  where  $u$ is the solution to the conductivity equation with $u |_{\partial \Omega} = f$.
  Let $\zeta_j = -i\tau_{}  (e_j + i e_n)$ for $1 \leqslant j < n$, and $\zeta_n =
  -i\tau_{}  (e_n + i e_1)$ where $\{ e_j \}_{j = 1}^n$ denotes the standard basis of $\R^n$.
   Then  $\sqrt\gamma(x) \nabla u (x)$ can be recovered from
  \begin{equation}
   \mathcal F(\sqrt\gamma\nabla u) ( k) = A^{- 1} [\lim_{\tau \rightarrow \infty} s ( k,
    \zeta_1), \ldots, \lim_{\tau \rightarrow \infty} s ( k, \zeta_n)]^T,
    \label{eq:F}
  \end{equation}
  where $A$ is a known invertible matrix and $\mathcal F$ denotes the Fourier transform.
\end{proposition}
Set $\hat\zeta=\zeta/|\zeta|$. Note that substituting (\ref{eq:mf-approx2}) in the definition of $s ( k,
\zeta)$ yields
\begin{equation}
  \begin{array}{lll}
    & s ( k, \zeta) & =i\displaystyle\int_{\Omega} e^{- i  k \cdot x} {\sqrt\gamma} \nabla u
    \cdot \hat{\zeta} ~\mathrm{d} x + O (| \zeta |^{- 1}) .\\
    &  & = i\mathcal F({\sqrt\gamma} \nabla u \cdot \hat{\zeta}) ( k) + O (| \zeta
    |^{- 1})
  \end{array} \label{kucuk-s}
\end{equation}
which in turn can be used to obtain (\ref{eq:F}). 

The calculation of $s ( k, \zeta)$
requires the knowledge of  $g=u_\zeta|_{\partial\Omega}$. We refer to  {\cite{RegularizedDbar}} and {\cite{Nachman1988}} where these traces of CGO solutions can be recovered from the DtN map $\Lambda_\gamma$.
Next, assuming the knowledge of this internal data
${\sqrt\gamma} (x) \nabla u (x)$ in $\Omega$, AET is reduced to an inverse
problem of reconstructing conductivity $\gamma$ from this internal data, which can be solved stably.

\begin{theorem}{\cite{ilker2012}}
  \label{thm-sta-pde}
  Let $\Omega$ be a bounded domain in $\R^n$ with smooth boundary. Given two  conductivities $\gamma$ and $\tilde\gamma$, and two boundary potentials $f$ and $\tilde f\in H^{1/2}(\partial\Omega)$ satisfying
  \[ \| f \|_{H^{1 / 2} (\partial \Omega)}, \| \tilde{f} \|_{H^{1 / 2}
     (\partial \Omega)} \leqslant M \]
     for some positive constant $M$, 
  denote the corresponding internal data $\mathsf{H} =\sqrt\gamma\nabla u$ and $\tilde{\mathsf{H}}
  =\sqrt{\tilde\gamma} \nabla \tilde{u}$, respectively. Then we have the following
  stability estimates:
  \begin{enumerate}
    \item Let ${\mathsf{H}_{\min} (x) = \min \{ | \mathsf{H} (x) |^{- 1}, | \tilde{\mathsf{H}} (x) |^{-
    1} \}}$. If $\mathsf{H}_{\min} \in L^2 (\Omega)$, then there exists a constant {$C
    = C (M, \Omega, \|
         \mathsf{H}_{\min} \|_{L^2 (\Omega)})$} such that
    \[
         \| \sqrt{\tilde\gamma} -\sqrt{\gamma} \|_{L^1 (\Omega)} \leqslant C (\| \tilde{\mathsf{H}} - \mathsf{H}\|_{H^1 (\Omega)} + \|
         \tilde{f} - f \|_{H^{1 / 2} (\partial \Omega)}).
       \]
    \item For some $l \geqslant 1$, suppose that  $\sqrt{\gamma},\sqrt{\tilde\gamma}$ are bounded by $M$ in  $C^{l,\alpha} (\bar{\Omega})$. Similarly, suppose $ f, \tilde{f} $ are bounded by $M$ in $C^{l, \alpha}
    (\partial \Omega)$. If $| \mathsf{H} (x) | > \delta > 0$,
    then there exists a constant $C = C (M, \Omega)$ such that
    \[
         \| \sqrt{\tilde\gamma} -\sqrt{\gamma} \|_{C^{l - 1, \alpha}}  \leqslant C
         (\| \tilde{\mathsf{H}} - \mathsf{H} \|_{C^{l - 1, \alpha} (\bar{\Omega})} + \|
         \tilde{f} - f \|_{C^{l, \alpha} (\partial \Omega)}) .
        \]
  \end{enumerate}
\end{theorem}

It is also possible to extend the above theorem to the case with a more general internal
data of the form $\mathsf{H} = \gamma^s \nabla u$ for $s\in \R$. In \cite{ilker2012} this is achieved by using similar techniques as in {\cite{BU2010}}.

Another internal data used in AET is power densities of the form
$\gamma | \nabla u |^2$ 
and the second step is now to recover the conductivity from these power densities. CGO\ solutions are also useful in this case. In
{\cite{BBMT2011}} a similar  inverse problem  is studied
where the power densities are considered as   the  internal data. The authors give stability estimates and
reconstruction of $\gamma$ using a data set of $m$ illuminations (where $m$ is equal to the spatial dimension $n$ for even $n$ and $n+1$ otherwise).
The corresponding set of internal data is assumed to satisfy
\begin{equation}
  \det (\mathsf{H}_1, \ldots, \mathsf{H}_m) \geqslant c_0 > 0 \label{req:F-away-0}
\end{equation}
in $\Omega$. The set of solutions that satisfy the above condition  is a key tool in their approach and it is  constructed by utilizing
CGO solutions in the Lemma $3.3$ of {\cite{BBMT2011}}.

\begin{lemma}{\cite{BBMT2011}} Let $n = 3$ and $\gamma\in H^{\frac{9}{2} +
  \epsilon} (\Omega)$ for some $\epsilon > 0$, be bounded from below by
  a positive constant. Then there exists a non-empty open set $\mathcal{G}
  \subset \left( H^{\frac{1}{2}} (\partial \Omega) \right)^4$ of quadruples of
  illuminations such that for any $g= (g_1, g_2, g_3, g_4) \in
  \mathcal{G}$, there exists an open cover of $\Omega$ of the form
  $\{\Omega_{2 i - 1}, \Omega_{2 i} \}_{1 \leq i \leq N}$ and a constant $c_0
  > 0$ such that
  \begin{equation}
    \inf_{x \in \Omega_{2 i - 1}} \det (\mathsf{H}_1, \mathsf{H}_2, \tilde{s}_i
    \mathsf{H}_4) \geqslant c_0  \; \tmop{and} \inf_{x \in \Omega_{2 i}} \det (\mathsf{H}_1,
    \mathsf{H}_2, s_i \mathsf{H}_3) \geqslant c_0,  1 \leqslant i
    \leqslant N, \label{det-away-0}
  \end{equation}
  where $\mathsf{H}=(\mathsf{H}_1, \mathsf{H}_2, \mathsf{H}_3, \mathsf{H}_4)$ is the internal data corresponding to the
  illumination $g$ and $s_i$,
  $\tilde{s}_i$ equal to $\pm 1$.
\end{lemma}

The existence of the set of illuminations whose internal data satisfies (\ref{det-away-0}) is  obtained by employing the fact
that the gradient of CGO solutions can be chosen to be  flat up to known
functions and terms that can be made negligible. In {\cite{BM2011Redundant}}
the authors generalize their results to multi-dimensional settings, using an internal data of the form
$\gamma^{\alpha} \nabla u_j \cdot \nabla u_k$, 
where $\alpha\in \mathbb{R},\; (n-2)\alpha+1\not=0$ and by constructing
CGO solutions to satisfy an analog of
(\ref{det-away-0}).
{In a series of work  \cite{BG2, BG, BGM3, BGM,bal-mon-2012,bal-mon-2013}, these results are extended to the case of anisotropic conductivities and also for the Maxwell's equations. }


\section{Quantitative Thermo-Acoustic Tomography}\label{sec:QTAT}

In this section, we consider the second step of the inversion for thermo-acoustic tomography (TAT). Recall that TAT is another coupled-physics imaging method implementing the photo-acoustic effect of electromagnetic radiation. But different from PAT, lower frequency electromagnetic ``illuminations" are used, in order to achieve deeper penetration in the tissue. As a result, the underlying equation is either modeled by an approximate scalar Helmholtz type equation or the full Maxwell's system. We explore the applications of CGO type solutions in solving inverse problems for both models, that is, to recover the high contrast conductivity (otherwise known as the absorption coefficient) from the absorbed energy by the tissue. A difference of these inverse problems from the previous ones is that the internal measurements are no longer linear functionals of the solutions but the norm squares of them. Consequently, we can see that the CGO solutions are used in very different fashions in answering uniqueness and stability questions for the two models. In particular, the general uniqueness question is still open for the system case.

\subsection{QTAT for scalar equations}

In this section we consider the Helmholtz model of radiation propagation, governed by the equation
\begin{equation}\label{eqn:QTAT_scalar}
  \left\{
\begin{aligned}
(\Delta + q )u = & 0 , \quad  \mbox{ in } \Omega,
\\  u = & g , \quad  \mbox{ on } \partial\Omega,
\end{aligned}
\right.
\end{equation}
where $q(x) = k^2 + i k \sigma(x)$ with $k$ denoting the wavenumber, $\sigma(x)$ the conductivity we wish to reconstruct, and $g(x)$ the boundary illumination. The amount of absorbed radiation by the underlying medium is given by
\[\mathsf{H}(x)=\sigma(x)|u(x)|^2, \quad\quad x\in\Omega.\]
A slightly more detailed derivation of this scalar approximation of electromagnetic propagation can be found in \cite{BRUZ}.

In \cite{BRUZ}, the CGO solutions to \eqref{eqn:QTAT_scalar} are used to show that the internal data $H$ leads to a functional of $\sigma$ that admits a unique fixed point. To be more specific, we take $\sigma\in H^s(\Omega)$ for $s>n/2$, so $q\in H^s(\Omega)$ and we have from Section \ref{sec:QPAT} the smoother CGO solutions
\begin{equation}\label{eqn:QTAT_CGO}
u = e^{ i \zeta \cdot x}( 1 + \psi_\zeta),
\end{equation}
where $\psi_\zeta$ solves
\begin{equation}\label{eqn:remainder}
\Delta \psi_\zeta + 2 i \zeta \cdot \nabla \psi_\zeta = - q(x) (1 + \psi_\zeta), \mbox { in } \mathbb{R}^n,
\end{equation}
and satisfies the following estimate, by \eqref{ineq_sobolev},
\begin{equation}\label{eqn:psi_est}
\|\psi_\zeta\|_{H^s(\Omega)} \leq \frac{C}{|\zeta|} \|q\|_{H^s(\Omega)}.
\end{equation}
Moreover, by refining the Neumann series argument step in the construction of these solutions, it is shown in \cite{BRUZ} that $\psi_\zeta$ is actually a contraction of $\sigma$, that is, for $|\zeta|$ large enough,
\begin{equation}\label{eqn:psi_contr}
\|\psi_\zeta-\tilde{\psi}_\zeta\|_{H^s(\Omega)}\leq\frac C{|\zeta|}\|\sigma-\tilde\sigma\|_{H^s(\Omega)},\quad s>n/2,
\end{equation}
where $\tilde\psi_\zeta$ is the solution to \eqref{eqn:remainder} with $\sigma$ replaced by $\tilde\sigma$.

Substituting the CGO solutions into the internal data, we obtain
\begin{equation}
e^{- i (\zeta + \bar{\zeta}) \cdot x}\mathsf{H}(x) = \sigma(x) + \mathcal{H}[\sigma](x)
\end{equation}
with
\begin{equation}\label{TH}
\mathcal{H}[\sigma] = \sigma (\psi_\zeta + \psi_{\bar{\zeta}} + \psi_\zeta \psi_{\bar{\zeta}}).
\end{equation}
Using estimates \eqref{eqn:psi_est} and \eqref{eqn:psi_contr}, and the fact that $H^s(\Omega)$ is an algebra for $s>n/2$, we have shown that for $|\zeta|$ sufficiently large, there exists a constant $\mathfrak c<1$ such that
\begin{equation}
\|\mathcal H[\sigma]-\mathcal H[\tilde\sigma]\|_{H^s(\Omega)}\leq\mathfrak c\|\sigma-\tilde\sigma\|_{H^s(\Omega)},
\end{equation}
that is, $\mathcal H[\sigma]$ is a contraction in $H^s(\Omega)$.
To summarize, we have
\begin{theorem}{\cite{BRUZ}} Let $\zeta\in\C^n$ be such that $|\zeta|$ is sufficiently large and $\zeta\cdot\zeta=0$. Let $\sigma$ be a function in
\[\mathcal M:=\left\{\sigma\in H^s(\Omega): s>n/2, \|\sigma\|_{H^s(\Omega)}\leq M<\infty\right\}.\]
Then there exists a boundary illumination $g=u|_{\partial\Omega}$ where $u$ is the CGO solution \eqref{eqn:QTAT_CGO} such that the corresponding internal measurement $\mathsf{H}(x)$ uniquely determines $\sigma$. Moreover, we have the reconstruction algorithm
\[\sigma=\lim_{m\rightarrow\infty}\sigma_m,\quad\sigma_0=0,\quad\sigma_m=e^{- i (\zeta+\overline\zeta)\cdot x}\mathsf{H}(x)-\mathcal{H}[\sigma_{m-1}](x),\]
where $\mathcal H[\sigma]$ is defined by \eqref{TH}.
Let $\tilde{\mathsf{H}}$ be the counterpart replacing $\sigma$ by $\tilde{\sigma}\in\mathcal M$. Then there exists a constant $C$ independent of $\sigma$ and $\tilde\sigma$ such that
\[\|\sigma-\tilde\sigma\|_{H^s(\Omega)}\leq C\|\mathsf{H}-\tilde{\mathsf{H}}\|_{H^s(\Omega)}.\]

\end{theorem}

The remainder part of the section is original. By an estimate presented in \cite{Serov2012}, we are able to reduce the regularity assumption on the coefficient $\sigma$ to $L^\infty$ in the uniqueness and stability result above.

\begin{lemma}{ \cite[Lemma~1]{Serov2012}}\label{lemma:L-infty-est}
Let $G_\zeta$ be the Faddeev's Green kernel defined in \eqref{eqn_Faddeev}. There exists constant $c>0$ depending on $\sigma$ such that for any $f\in L^2(\Omega)$ and for $|\zeta|>1$, we have that
\[
\|G_\zeta f\|_{L^\infty(\mathbb{R}^n)} \leq \frac{c}{|\zeta|^l} \|f\|_{L^2(\Omega)},
\]
where $l <1$ for $n=2$ and $l <\frac{1}{2}$ for $n=3$.
\end{lemma}

\begin{theorem}
Let $\sigma,\tilde{\sigma}$ be a bounded measurable function satisfying
\[
 \|k^2 + i k \sigma\|_{L^\infty(\Omega)} \leq M,
 \]
 for some constant $M>0$. Assume that $\zeta\in\mathbb{C}^n$ satisfies that $\zeta \cdot \zeta =0$ and $|\zeta|$ is sufficiently large.
Then there exists a boundary illumination $g$ such that the corresponding internal measurement $\mathsf{H}(x)$ uniquely determines $\sigma$. Moreover, we have the reconstruction algorithm
\begin{equation}\label{eqn:L_infty}\sigma=\lim_{m\rightarrow\infty}\sigma_m,\quad\sigma_0=0,\quad\sigma_m=e^{- i (\zeta+\overline\zeta)\cdot x}\mathsf{H}(x)-\mathcal{H}[\sigma_{m-1}](x),\end{equation}
where $\mathcal H$ is defined by \eqref{TH}.
Let $\tilde{\mathsf{H}}$ be the counterpart replacing $\sigma$ by $\tilde{\sigma}$, also measurable and satisfying \eqref{eqn:L_infty}. Then there exists a constant $C$ independent of $\sigma$ and $\tilde\sigma$ such that
\begin{equation}
  \label{stab-est}
  \|\sigma - \tilde{\sigma}\|_{L^\infty(\Omega)} \leq C \|\mathsf{H} - \tilde{\mathsf{H}}\|_{L^\infty(\Omega)}
  \end{equation}
  holds true.
\end{theorem}

\begin{proof}
In the following, $C$ denotes the generic positive constant. Note that, by using the estimate of the operator $(\Delta+2 i \zeta\cdot\nabla)^{-1}$ as in Proposition~\ref{prop:oper-est}, we have that $\psi_\zeta$ is $L^2$-bounded for sufficiently large $|\zeta|$.
  Applying Lemma~\ref{lemma:L-infty-est} to $\psi_\zeta$, with the aid of the $L^2$-boundedness of $\psi_\zeta$, we obtain that
\begin{equation}
\label{est-1}
  \begin{aligned}
  \|\psi_\zeta\|_{L^\infty(\mathbb{R}^n)} \leq  & \frac{C}{|\zeta|^l}\|q (1 + \psi_\zeta)\|_{L^2(\Omega)}
  \\
  \leq  & \frac{C}{|\zeta|^l} \|q\|_{L^\infty(\Omega)} \|1 + \psi_\zeta\|_{L^2(\Omega)}
  \\
  \leq & \frac{C}{|\zeta|^l},
\end{aligned}
\end{equation}
where $C$ only depends on $M$ and $\Omega$ for $|\zeta|$ large.
Similarly, we can obtain the boundedness of $\psi_{\bar{\zeta}}$, $\tilde{\psi}_{\zeta}$ and $\tilde{\psi}_{\bar{\zeta}}$.

Let $v= \psi_\zeta - \tilde{\psi}_\zeta$. It follows that
\[
-(\Delta  + 2 i \zeta \cdot \nabla) v = -q(1+\psi_\zeta) + \tilde{q}(1+\tilde{\psi}_\zeta) = (\tilde{q} - q)(1+\tilde{\psi}_\zeta) - q (\psi_\zeta -\tilde{\psi}_\zeta).
\]
Applying Lemma~\ref{lemma:L-infty-est} to $v$, we obtain that
\[
\begin{aligned}
\|v\|_{L^\infty(\mathbb{R}^n)} \leq  & \frac{C}{|\zeta|^l} \|(\tilde{q} - q)(1+\tilde{\psi}_\zeta) - q v\|_{L^2(\Omega)}
\\
\leq  & \frac{C}{|\zeta|^l} \left(\|\tilde{q} - q\|_{L^\infty(\Omega)} \|1+\tilde{\psi}_\zeta\|_{L^2(\Omega)}
+ M|\Omega|^{1/2} \|v\|_{L^\infty(\Omega)} \right).
\end{aligned}
\]
The above inequality implies that
\begin{equation}
\label{est-2}
\|\psi_\zeta - \tilde{\psi}_\zeta\|_{L^\infty(\Omega)} \leq  \frac{C}{|\zeta|^l} \|q - \tilde{q} \|_{L^\infty(\Omega)},
\end{equation}
for some constant $C$ only depends on $M$ and $\Omega$ when $|\zeta|$ is large enough.
The same estimate applies to $\psi_{\bar{\zeta}} - \tilde{\psi}_{\bar{\zeta}}$.
It follows that
\begin{equation}
\label{est-4}
\begin{aligned}
& \|\psi_\zeta \psi_{\bar{\zeta}} - \tilde{\psi}_\zeta \tilde{\psi}_{\bar{\zeta}}\|_{L^\infty(\Omega)}
\\
\leq & \|\psi_{\bar{\zeta}} \|_{L^\infty(\Omega)} \|\psi_\zeta - \tilde{\psi}_\zeta\|_{L^\infty(\Omega)} + \|\tilde{\psi}_{\zeta} \|_{L^\infty(\Omega)}\| \psi_{\bar{\zeta}} - \tilde{\psi}_{\bar{\zeta}}\|_{L^\infty(\Omega)}
\\
\leq &\frac{C}{|\zeta|^{2l}} \|q - \tilde{q} \|_{L^\infty(\Omega)}.
\end{aligned}
\end{equation}
Now, we observe that
\begin{equation}
\label{rem-eqn}
\begin{aligned}
 \mathcal{H}[\sigma] - \mathcal{H}[\tilde{\sigma}] =  & \sigma (\psi_\zeta + \psi_{\bar{\zeta}} + \psi_\zeta \psi_{\bar{\zeta}}) - \tilde{\sigma} (\tilde{\psi}_\zeta + \tilde{\psi}_{\bar{\zeta}} + \tilde{\psi}_\zeta \tilde{\psi}_{\bar{\zeta}})
  \\
  = & \sigma (\psi_\zeta -\tilde{\psi}_\zeta ) + (\sigma - \tilde{\sigma}) \tilde{\psi}_\zeta + \sigma (\psi_{\bar{\zeta}} -\tilde{\psi}_{\bar{\zeta}} ) + (\sigma - \tilde{\sigma}) \tilde{\psi}_{\bar{\zeta}}
  \\
  & + \sigma (\psi_\zeta \psi_{\bar{\zeta}} - \tilde{\psi}_\zeta \tilde{\psi}_{\bar{\zeta}}) + (\sigma - \tilde{\sigma}) \tilde{\psi}_\zeta \tilde{\psi}_{\bar{\zeta}}.
\end{aligned}
\end{equation}
By \eqref{est-1}, \eqref{est-2} and \eqref{est-4}, we conclude that
\begin{equation}
\label{contrc}
\|\mathcal{H}[\sigma] - \mathcal{H}[\tilde{\sigma}] \|_{L^\infty(\Omega)} \leq C \left(\frac{1}{|\zeta|^l} +\frac{1}{|\zeta|^{2l}} \right)\|\sigma - \tilde{\sigma} \|_{L^\infty(\Omega)},
\end{equation}
for some constant $C$ only depends on $k$, $M$ and $\Omega$. This yields that $\mathcal{H}$ is a contraction when one chooses a sufficiently large $|\zeta|$.
Recall that
\[
  \mathsf{H} - \tilde{\mathsf{H}} = \sigma - \tilde{\sigma} + \mathcal{H}[\sigma] - \mathcal{H}[\tilde{\sigma}] .
\]
The stability estimate \eqref{stab-est} follows from \eqref{contrc}.
\end{proof}

We formulate the stability estimate in $L^\infty$, which is an algebra as $H^s$. Let us remark however that stability estimate in $L^2$ as
\begin{equation}
  \|\sigma - \tilde{\sigma}\|_{L^2(\Omega)} \leq C \|\mathsf{H} - \tilde{\mathsf{H}}\|_{L^2(\Omega)}
\end{equation}
could also be obtained with few adjustments of the proof. Actually, with the $L^2$ version of \eqref{est-2}, \eqref{est-4} and \eqref{contrc}, the above estimate follows.

\subsection{QTAT for Maxwell system}

The above scalar model is an approximation of the full Maxwell system model of electromagnetic radiation. Consider time harmonic electromagnetic waves satisfying
\begin{equation}\label{eqn:QTAT_Maxwell}
-\nabla\times\nabla\times E+\left(k^2n(x)+ik\sigma(x)\right)E=0\quad\textrm{in }\Omega
\end{equation}
with boundary illumination given in terms of the tangential electric field
\[\nu\times E|_{\partial\Omega}=g,\]
where $\nu$ is the unit outer normal to $\partial\Omega$. The amount of absorbed radiation by the underlying tissue is given by
\[\mathsf{H}_g(x)=\sigma(x)|E(x)|^2\quad\textrm{ for }x\in\Omega.\]
The quantitative step of TAT concerns the reconstruction of $(n,\sigma)$ from knowledge of
$\{\mathsf{H}_j(x)=\mathsf{H}_{g_j}(x)\}_{1\leq j\leq J}$ obtained from the first step by probing the medium with $J$ illuminations $\{g_j\}_{1\leq j\leq J}$. It was shown in \cite{BRUZ} that with the refractive index $n(x)$ being a constant, the conductivity $\sigma(x)$ can be uniquely and stably reconstructed from a single (well-chosen) internal measurement provided that $\sigma$ is sufficiently small (compared to $k$).
Here we review the results in \cite{BZ} for the general case. 
The fixed point type analysis above is no longer available due to lack of a contraction estimate as \eqref{eqn:psi_contr} for the CGO solutions to Maxwell's equations. Alternatively, in \cite{BZ},
it is shown that the linearization of the propagation equation and the internal measurements $\{\mathsf{H}_j\}$ as an operator of the electric fields and the parameters, is an elliptic matrix-valued differential operator. The ellipticity is shown by plugging in proper CGO solutions to Maxwell's equations. Therefore, with sufficiently many measurements, $\left(n(x),\sigma(x)\right)$ can be uniquely and stably reconstructed with no loss of derivatives. \\

Set $q(x)=k^2n(x)+ik\sigma(x)$. Notice that from \eqref{eqn:QTAT_Maxwell}, we have $\nabla\nabla\cdot(qE)=0$. Then \eqref{eqn:QTAT_Maxwell} can be rewritten as
\begin{equation}\label{eqn:ellip}\left(\Delta+\frac 1 q[\nabla\nabla\cdot,q]+q\right)E=0, \end{equation}
where $[A, B]:=AB-BA$ is the usual commutator.
We denote by $\tilde{\mathscr F}(v)=\mathscr H$ the following nonlinear system
\[\left\{\begin{array}{rl}
\left(\Delta+\displaystyle\frac 1 q[\nabla\nabla\cdot,q]+q\right)\tilde E_j=&0,\\
\left(\Delta+\displaystyle\frac 1 {\overline q}[\nabla\nabla\cdot,\overline q]+\overline q\right)\overline{\tilde E_j}=&0,\\
\Delta(\sigma|\tilde E_j|^2)=&\Delta \mathsf{H}_j.
\end{array} \quad 1\leq j\leq J\right.\]
of $v=\left(\{\tilde E_j, \overline{\tilde E_j}\}_{j=1}^J,\sigma, n\right)$, where $\tilde E_j$ is the solution corresponding to the boundary illumination $g_j$. 

It is not hard to derive that the linearization of \eqref{eqn:ellip} is
\begin{equation}\label{eqn:linear_1}\left(\Delta+\frac 1 {q_0}[\nabla\nabla\cdot,q_0]+q_0\right)\delta_{E_j}=-\delta_q E_j-\frac 1 {q_0}\nabla\nabla\cdot(\delta_qE_j).\end{equation}
Here $\delta_q=k^2\delta_n+ik\delta_\sigma$ denotes the perturbation of $q_0=k^2n_0+ik\sigma_0$, that is, $q=q_0+\epsilon\delta_q$ for $\epsilon>0$ small, and $\tilde E_j=E_j+\epsilon\delta_{E_j}+o(\epsilon)$ where $E_j$ is the solution corresponding to $q_0$.
Also, taking the Fr\'echet derivative of $\Delta\mathsf{H}_j$ (equivalent to taking $\Delta$ of the Fr\'echet derivative $\mathfrak d\mathsf{H}_j$) yeilds
\begin{equation}\label{eqn:linear_2}
|E_j|^2\Delta\delta_\sigma-\frac{\sigma_0}{q_0}\overline{E_j}\cdot(E_j\cdot\grad\grad\delta_q)-\frac{\sigma_0}{\overline q_0}E_j\cdot(\overline{E_j}\cdot\grad\grad\overline{\delta_q})+\textrm{l.o.t.}=\Delta\mathfrak d \mathsf{H}_j,\end{equation}
where l.o.t. represents the lower order terms.

Therefore, we obtain a system of linear equations: \eqref{eqn:linear_1}, the conjugate of \eqref{eqn:linear_1} and \eqref{eqn:linear_2} for $1\leq j\leq J$, in the form $\mathcal A(x,D) w=\mathcal S$ where
\begin{equation}\label{eqn:X-b}
\begin{split}
w&=\left({\delta E}_1, {\delta E}^*_1, \ldots, {\delta E}_J, {\delta E}^*_J, \delta_\sigma, \delta_n\right)^T,\\
\mathcal S&=\left(0,\ldots, 0,\Delta\mathfrak d \mathsf{H}_1,\ldots,\Delta\mathfrak d \mathsf{H}_J\right)^T,
\end{split}\end{equation}
and $\mathcal A(x,D)$ is a second order $7J\times (6J+2)$ matrix differential operator with the principal part in the Douglis-Nirenberg sense given by
\begin{equation}\label{eqn:A0}\mathcal A_0(x,D)=\left(\begin{array}{c|c}\Delta I_{6J} & \mathcal A_{12}(x,D)  \\ \hline 0 & \begin{array}{cc}a_1(x,D) & b_1(x,D) \\ \vdots & \vdots \\a_J(x,D) & b_J(x,D)\end{array} \end{array}\right),\end{equation}
where $I_k$ is the $k\times k$ identity matrix. Here $a_j(x,D)$ and $b_j(x,D)$ are second order operators whose symbols are
\[a_j(x,\xi)= -|E_j|^2|\xi|^2+\frac{2k^2\sigma_0^2}{|q_0|^2}|E_j\cdot\xi|^2,\qquad b_j(x,\xi)=\frac{2k^4\sigma_0n_0}{|q_0|^2}|E_j\cdot\xi|^2.\]
We remark that at this point the reason becomes self-explained for taking the laplacian of the internal function $\mathsf{H}_j$ in our nonlinear system.
Then for $\mathcal A_0(x,\xi)$ to have full rank $6J+2$ (Here $J\geq2$ so the system is not underdetermined) for every $x\in\Omega$ and $\xi\in\R^3\backslash\{0\}$, one has to show that the rank of
\[\mathcal A_{22}(x,\xi):=\left(\begin{array}{cc}a_1(x,\xi) & b_1(x,\xi) \\ \vdots & \vdots \\a_J(x,\xi) & b_J(x,\xi)\end{array}\right)\]
is 2. This is equivalent to show the following relation for every $x\in\Omega$:
\begin{equation}\label{eqn:ellip-cond-1}
|E_j|^2|E_l|^2\left(|\widehat E_j\cdot\xi|^2-|\widehat E_l\cdot\xi|^2\right)=0\quad 1\leq j<l\leq J\quad \Rightarrow\quad \xi=0
\end{equation}
where $\widehat E_j:=E_j/|E_j|$.

To this end, we implement CGO solutions to the background Maxwell's equations \eqref{eqn:QTAT_Maxwell}  with $q$ replaced by the background $q_0$. These solutions were also used in ESC as mentioned in Proposition \ref{thm:CGO2}, originally constructed in \cite{CP}. Basically, Faddeev kernel allows the construction in higher order Sobolev spaces, which in turn provides an $L^\infty$ bound at our disposal.

Let $q_0=k^2n_0+ik\sigma_0\in H^s(\R^n)$ $(s>n/2+2)$ and $q_0(x)\neq0$ everywhere in $\Omega$. Suppose that $n_0(x)-n_c\geq 0$ and $\sigma_0(x)\geq0$ are compactly supported on some ball $B\supset\supset\overline\Omega$ for some constant $n_c>0$. Denote $\kappa=k\sqrt{n_c}$. 
For $\rho, \rho^\perp\in\mathbb S^{n-1}$ with $\rho\perp\rho^\perp$ and $\tau>0$, we choose
\[\zeta:=-i\tau\rho+\sqrt{\tau^2+\kappa^2}\rho^\perp\]
and
\[\eta_\zeta:=\frac{1}{|\zeta|}\left(-(\zeta\cdot\vec{a})\zeta-\kappa\zeta\times\vec{b}+k^2\vec{a}\right)\]
for any $\vec a, \vec b\in\C^n$. Then we have the unique CGO solution in $H^1_{\textrm{loc}}(\R^n)$ of the form
\[E_\zeta(x)=\gamma_0^{-1/2}e^{i\zeta\cdot x}\left(\eta_\zeta+R_\zeta(x)\right)\]
where $\gamma_0=q_0/\kappa^2$ and $\gamma_0^{1/2}$ denotes the principal branch. Moreover, $R_\zeta\in L^\infty(\Omega)$ satisfies
\begin{equation}\label{eqn:Linftybd}\|R_\zeta\|_{L^\infty(\Omega)}\leq C\end{equation}
where $C>0$ is independent of $\tau$.

To find enough CGO solutions to prove \eqref{eqn:ellip-cond-1}, first notice that as $\tau\rightarrow\infty$, by \eqref{eqn:Linftybd},
\[|E_\zeta(x)|\sim |\gamma_0|^{-1/2}e^{\tau \rho\cdot x}\sqrt 2\tau\neq 0\]
everywhere in $\Omega$ for a proper choice of $\vec a,\vec b$, and also
\[|\widehat E_\zeta(x)\cdot\xi|\sim \frac 1{\sqrt 2}\sqrt{|\xi\cdot\rho|^2+|\xi\cdot\rho^\perp|^2}\]
independent of $x\in\Omega$. Then \cite{BZ} prescribed a way to choose $(n+1)$ (where $n$ is the spatial dimension, hence $n=3$ in our case) pairs of $(\rho_j,\rho_j^\perp)$ such that at least two of $\{|\widehat E_{\zeta_j}\cdot\xi|\}_{j=1,\ldots,n+1}$ are not equal. Therefore, condition \eqref{eqn:ellip-cond-1} is fulfilled. Combined with regularity estimates, we have
\begin{lemma} \label{lem:ellip}
There exist $(n+1)$ CGO solutions $E_j=E_{\zeta_j}$ $(1\leq j\leq n+1)$ as above with $\tau$ sufficiently large such that we can find $(n+1)$ boundary illuminations $\{g_j\}$ from a neighborhood (in the $H^{s-1/2}(\partial\Omega)$ topology) of $\{E_j|_{\partial\Omega}\}$ to make sure the operator $\mathcal A(x, D)$ is elliptic for every $x\in\Omega$.
\end{lemma}

The next step is to augment the redundant elliptic linear system $\mathcal Aw=\mathcal S$ (in the bounded domain $\Omega$) with a {\it complementing boundary condition}, namely, a boundary condition satisfying the so-called Lopatinskii creterion (see \cite{Lop}). It is proved in \cite{BZ} that the Dirichlet boundary condition $w|_{\partial\Omega}$ satisfies the criterion, hence by Theorem 1.1 of \cite{Sol} together with Lemma \ref{lem:ellip} we have

\begin{theorem}{\cite{BZ}}
Let $\Omega$ be a bounded domain in $\R^3$ with smooth boundary. Set $J=4$. There exists a boundary illumination set $\{g_j\}_{j=1}^4$ such that the redundant linear system $\mathcal Aw=\mathcal S$ augmented with the Dirichlet boundary condition $w|_{\partial\Omega}=w^\delta$ is an elliptic boundary problem in the sense of \cite{Sol} (also known as Douglis-Nirenberg sense \cite{DN}). Moreover, we have the following Schauder type estimate
\begin{equation}\label{eqn:scha_linear}
\begin{split}
\|w\|_{H^s(\Omega)} \leq C_1\Big( \|\mathfrak d \mathsf H\|_{H^s(\Omega)} + \|w^\delta\|_{H^{s-\frac12}(\partial\Omega)}\Big) + C_2 \|w\|_{L^2(\Omega)},
\end{split}
\end{equation}
for all $s>2+\frac32=\frac72$ provided that $(\sigma,n,\{E_j\})$ are in $H^s(\Omega)$.
\end{theorem}

The estimate \eqref{eqn:scha_linear} implies that the linearized inverse problem is stable if the above linear boundary value problem is injective ($C_2=0$). For the nonlinear inverse problem, the strategy follows \cite{Bal} by considering the linear normal operator
\begin{equation}\label{eqn:linear_3}\mathcal A^t\mathcal A w=\mathcal A^t \mathcal S\quad\textrm{ in }\Omega,\end{equation}
for again the Dirichlet type boundary condition
\begin{equation}\label{eqn:linear_3b}w|_{\partial\Omega}=w^\delta\quad\textrm{and}\quad \partial_\nu w|_{\partial\Omega}=j^\delta.\end{equation}
It is shown in \cite{Bal} that the above linear problem is injective (i) when the coefficients $v=(\{E_j, E_j^*\}_{j=1}^J, \sigma,n)$ are in a sufficiently small vicinity of an analytic coefficient (with the vicinity depending on that analytic coefficients); and (ii) when the domain $\Omega$ is sufficiently small.

The stability estimate presented in the above theorem then extends to the following nonlinear inverse problem
\begin{equation}
\label{eq:nonlinfinalmodif}
  \mathscr F(v) := \mathcal A^t \tilde{\mathscr F}(v)   = \mathcal A^t \mathscr H\quad\mbox{ in }\quad\Omega,\qquad v|_{\partial \Omega} = v^\delta\ \mbox{ and } \ \partial_\nu v|_{\partial\Omega} = \mathfrak j^\delta.
\end{equation}
Then defining $v=v_0+w$ and linearizing the above inverse problem about $v_0$, we observe that the linear equation for $w$ is precisely of the form \eqref{eqn:linear_3}, \eqref{eqn:linear_3b}. This and the theory presented in \cite{Bal} allow us to obtain the following result.
\begin{theorem}
\label{thm:stabnonlin}
  Let us assume that the linear problem defined in \eqref{eqn:linear_3} and \eqref{eqn:linear_3b} is injective. Let $v$ and $\tilde v$ be solutions of \eqref{eq:nonlinfinalmodif} with respective source terms $\mathcal H$ and $\tilde{\mathcal H}$ and respective boundary conditions $v^\delta$ and $\tilde v^\delta$ as well as $\mathfrak j^\delta$ and $\tilde{\mathfrak j}^\delta$. Then $(\mathscr H,v^\delta,\mathfrak j^\delta)=(\tilde{\mathscr H},\tilde v^\delta,\tilde{\mathfrak j}^\delta)$ implies that $v=\tilde v$, in other words the nonlinear hybrid inverse problem is injective. Moreover, we have the stability estimate
  \begin{equation}
\label{eq:stabnon}
  \|v-\tilde v\|_{H^s(\Omega)} \leq C \Big( \|\mathscr H-\tilde {\mathscr H}\|_{H^s(\Omega)} + \|v^\delta-\tilde v^\delta\|_{H^{s-\frac12}(\partial\Omega)} +  \|\mathfrak j^\delta-\tilde{\mathfrak j}^\delta\|_{H^{s-\frac32}(\partial\Omega)} \Big).
\end{equation}
This estimate holds for $C=C_s$ when $s>\frac72$.
\end{theorem}

We refer the reader to \cite{Bal} for similar type of analysis applied to other coupled-physics imaging inverse problems.

\bibliographystyle{siam}
\bibliography{biblio}

\end{document}